\documentclass[letterpaper,onecolumn,superscriptaddress,11pte,shorttitle=Spiders,nopdfoutputerror,accepted=2021-11-22, issue=1, volume=4, shorttitle=papers/compositionality-4-1]{compositionalityarticle}

\usepackage{etex}

\usepackage[utf8]{inputenc}
\usepackage[english]{babel}
\usepackage[T1]{fontenc}
\usepackage{hyperref}

\usepackage{tikz-cd}
\usetikzlibrary{automata,arrows,positioning,decorations.pathreplacing,calc}
\tikzcdset{
  shift left/.default=+0.7ex,
  shift right/.default=+0.7ex}
\newenvironment{tikzar}[1][]{{}\kern-4pt\begin{tikzcd}[ampersand replacement=\&,#1]}%
{\end{tikzcd}\kern-4pt{}}
\tikzstyle{edge label
}=[font=\scriptsize]
\tikzcdset{every label/.append style = {font = \scriptsize}}
\tikzcdset{every node/.append style = {font = \scriptsize}}
\tikzcdset{arrow style = tikz,diagrams={>={Straight Barb[scale=0.8]}}}

\usepackage{rotating}

\usepackage{pstricks}
\usepackage{amssymb}
\usepackage{amsmath}
\usepackage{amsthm}
\usepackage{txfonts}
\usepackage{bbm}
\usepackage[only,bigcurlywedge,bigcurlyvee]{stmaryrd}
\usepackage{enumerate}
\usepackage{multicol}

\def\pushright#1{{
    \parfillskip=0pt            
    \widowpenalty=10000         
    \displaywidowpenalty=10000  
    \finalhyphendemerits=0      
    \leavevmode                 
    \unskip                     
    \nobreak                    
    \hfil                       
    \penalty50                  
    \hskip.2em                  
    \null                       
    \hfill                      
    {#1}                        
    \par}}                      

\def\qed{\pushright{$\square$}\penalty-700 \smallskip}
\newenvironment{prf}[1]{\begin{trivlist} \item[{\bf ~Proof}#1.]}%
{\qed\end{trivlist}}

\newcommand{\bpr}{\begin{prf}{}}
\newcommand{\epr}{\end{prf}}
\newcommand{\bprf}[1]{\begin{prf}{#1}}
\newcommand{\eprf}{\end{prf}}

\newtheorem{theorem}{Theorem}

\newtheorem*{corollary*}{Corollary}

\theoremstyle{remark}

\newdimen\proofrulebreadth \proofrulebreadth=.05em
\newdimen\proofdotseparation \proofdotseparation=1.25ex
\newdimen\proofrulebaseline \proofrulebaseline=2ex
\newcount\proofdotnumber \proofdotnumber=3
\let\then\relax
\def\hfi{\hskip0pt plus.0001fil}
\mathchardef\squigto="3A3B
\newif\ifinsideprooftree\insideprooftreefalse
\newif\ifonleftofproofrule\onleftofproofrulefalse
\newif\ifproofdots\proofdotsfalse
\newif\ifdoubleproof\doubleprooffalse
\let\wereinproofbit\relax
\newdimen\shortenproofleft
\newdimen\shortenproofright
\newdimen\proofbelowshift
\newbox\proofabove
\newbox\proofbelow
\newbox\proofrulename
\def\shiftproofbelow{\let\next\relax\afterassignment\setshiftproofbelow\dimen0 }
\def\shiftproofbelowneg{\def\next{\multiply\dimen0 by-1 }%
\afterassignment\setshiftproofbelow\dimen0 }
\def\setshiftproofbelow{\next\proofbelowshift=\dimen0 }
\def\setproofrulebreadth{\proofrulebreadth}

\def\prooftree{
\ifnum  \lastpenalty=1
\then   \unpenalty
\else   \onleftofproofrulefalse
\fi
\ifonleftofproofrule
\else   \ifinsideprooftree
        \then   \hskip.5em plus1fil
        \fi
\fi
\bgroup
\setbox\proofbelow=\hbox{}\setbox\proofrulename=\hbox{}%
\let\justifies\proofover\let\leadsto\proofoverdots\let\Justifies\proofoverdbl
\let\using\proofusing\let\[\prooftree
\ifinsideprooftree\let\]\endprooftree\fi
\proofdotsfalse\doubleprooffalse
\let\thickness\setproofrulebreadth
\let\shiftright\shiftproofbelow \let\shift\shiftproofbelow
\let\shiftleft\shiftproofbelowneg
\let\ifwasinsideprooftree\ifinsideprooftree
\insideprooftreetrue
\setbox\proofabove=\hbox\bgroup$\displaystyle 
\let\wereinproofbit\prooftree
\shortenproofleft=0pt \shortenproofright=0pt \proofbelowshift=0pt
\onleftofproofruletrue\penalty1
}

\def\eproofbit{
\ifx    \wereinproofbit\prooftree
\then   \ifcase \lastpenalty
        \then   \shortenproofright=0pt  
        \or     \unpenalty\hfil         
        \or     \unpenalty\unskip       
        \else   \shortenproofright=0pt  
        \fi
\fi
\global\dimen0=\shortenproofleft
\global\dimen1=\shortenproofright
\global\dimen2=\proofrulebreadth
\global\dimen3=\proofbelowshift
\global\dimen4=\proofdotseparation
\global\count255=\proofdotnumber
$\egroup  
\shortenproofleft=\dimen0
\shortenproofright=\dimen1
\proofrulebreadth=\dimen2
\proofbelowshift=\dimen3
\proofdotseparation=\dimen4
\proofdotnumber=\count255
}

\def\proofover{
\eproofbit 
\setbox\proofbelow=\hbox\bgroup 
\let\wereinproofbit\proofover
$\displaystyle
}%
\def\proofoverdbl{
\eproofbit 
\doubleprooftrue
\setbox\proofbelow=\hbox\bgroup 
\let\wereinproofbit\proofoverdbl
$\displaystyle
}%
\def\proofoverdots{
\eproofbit 
\proofdotstrue
\setbox\proofbelow=\hbox\bgroup 
\let\wereinproofbit\proofoverdots
$\displaystyle
}%
\def\proofusing{
\eproofbit 
\setbox\proofrulename=\hbox\bgroup 
\let\wereinproofbit\proofusing
\kern0.3em$
}

\def\endprooftree{
\eproofbit 
  \dimen5 =0pt
\dimen0=\wd\proofabove \advance\dimen0-\shortenproofleft
\advance\dimen0-\shortenproofright
\dimen1=.5\dimen0 \advance\dimen1-.5\wd\proofbelow
\dimen4=\dimen1
\advance\dimen1\proofbelowshift \advance\dimen4-\proofbelowshift
\ifdim  \dimen1<0pt
\then   \advance\shortenproofleft\dimen1
        \advance\dimen0-\dimen1
        \dimen1=0pt
        \ifdim  \shortenproofleft<0pt
        \then   \setbox\proofabove=\hbox{%
                        \kern-\shortenproofleft\unhbox\proofabove}%
                \shortenproofleft=0pt
        \fi
\fi
\ifdim  \dimen4<0pt
\then   \advance\shortenproofright\dimen4
        \advance\dimen0-\dimen4
        \dimen4=0pt
\fi
\ifdim  \shortenproofright<\wd\proofrulename
\then   \shortenproofright=\wd\proofrulename
\fi
\dimen2=\shortenproofleft \advance\dimen2 by\dimen1
\dimen3=\shortenproofright\advance\dimen3 by\dimen4
\ifproofdots
\then
        \dimen6=\shortenproofleft \advance\dimen6 .5\dimen0
        \setbox1=\vbox to\proofdotseparation{\vss\hbox{$\cdot$}\vss}%
        \setbox0=\hbox{%
                \advance\dimen6-.5\wd1
                \kern\dimen6
                $\vcenter to\proofdotnumber\proofdotseparation
                        {\leaders\box1\vfill}$%
                \unhbox\proofrulename}%
\else   \dimen6=\fontdimen22\the\textfont2 
        \dimen7=\dimen6
        \advance\dimen6by.5\proofrulebreadth
        \advance\dimen7by-.5\proofrulebreadth
        \setbox0=\hbox{%
                \kern\shortenproofleft
                \ifdoubleproof
                \then   \hbox to\dimen0{%
                        $\mathsurround0pt\mathord=\mkern-6mu%
                        \cleaders\hbox{$\mkern-2mu=\mkern-2mu$}\hfill
                        \mkern-6mu\mathord=$}%
                \else   \vrule height\dimen6 depth-\dimen7 width\dimen0
                \fi
                \unhbox\proofrulename}%
        \ht0=\dimen6 \dp0=-\dimen7
\fi
\let\doll\relax
\ifwasinsideprooftree
\then   \let\VBOX\vbox
\else   \ifmmode\else$\let\doll=$\fi
        \let\VBOX\vcenter
\fi
\VBOX   {\baselineskip\proofrulebaseline \lineskip.2ex
        \expandafter\lineskiplimit\ifproofdots0ex\else-0.6ex\fi
        \hbox   spread\dimen5   {\hfi\unhbox\proofabove\hfi}%
        \hbox{\box0}%
        \hbox   {\kern\dimen2 \box\proofbelow}}\doll%
\global\dimen2=\dimen2
\global\dimen3=\dimen3
\egroup 
\ifonleftofproofrule
\then   \shortenproofleft=\dimen2
\fi
\shortenproofright=\dimen3
\onleftofproofrulefalse
\ifinsideprooftree
\then   \hskip.5em plus 1fil \penalty2
\fi
}

\renewcommand{\to}{\xrightarrow{}}
\newcommand{\tto}[1]{\xrightarrow{#1}}

\newcommand{\prel}{\searrow}

\newcommand{\der}[1]{\left(\xleftarrow{#1}\right)}

\newcommand{\derr}[2]{\left({#1}\leftarrow{#2}\right)}

\newcommand{\derrr}[3]{\left({#1}\xleftarrow{#2}{#3}\right)}

\makeatletter
\newcommand\xleftrightarrow[2][]{%
  \ext@arrow 9999{\longleftrightarrowfill@}{#1}{#2}}
\newcommand\longleftrightarrowfill@{%
  \arrowfill@\leftarrow\relbar\rightarrow}
\makeatother

\newcommand{\dereq}[1]{\left(\xleftrightarrow{#1}\right)}

\newcommand{\derreq}[2]{\left({#1}\leftrightarrow{#2}\right)}

\newcommand{\derrreq}[3]{\left({#1}\xleftrightarrow{#2}{#3}\right)}

\newcommand{\WP}{\mbox{\Large $\wp$}}
\newcommand{\lcomp}{\ ;}

\newcommand{\scomp}[2]{\left({#1}\lcomp{#2}\right)}

\newcommand{\Rel}{\mathsf{Rel}}
\newcommand{\PRel}{\mathsf{PRel}}
\newcommand{\Preord}{\mathsf{Prd}}

\newcommand{\id}{\mathrm{id}}

\newcommand{\CCC}{{\mathcal C}}

\newcommand{\EEE}{{\mathcal E}}

\newcommand{\JJJ}{{\mathcal J}}

\newcommand{\LLL}{{\mathcal L}}
\newcommand{\MMM}{{\mathcal M}}

\newcommand{\OOO}{{\mathcal O}}

\newcommand{\ZZZ}{{\mathcal Z}}

\newcommand{\NNn}{{\mathbb N}}

\newcommand{\ZZz}{{\mathbb Z}}

\newcommand{\OOne}{\mathbbm{1}}

\mathcode`\<="4268 
\mathcode`\>="5269 
\mathchardef\gt="313E 
\mathchardef\lt="313C 

\newcommand{\beq}{\begin{equation}}
\newcommand{\eeq}{\end{equation}}
\newcommand{\ba}[1]{\begin{array}{#1}}
\newcommand{\ea}{\end{array}}
\newcommand{\bea}{\begin{eqnarray}}
\newcommand{\eea}{\end{eqnarray}}
\newcommand{\bear}{\begin{eqnarray*}}
\newcommand{\eear}{\end{eqnarray*}}

\newcommand{\Up}[1]{{\Uparrow}{#1}}
\newcommand{\Do}[1]{{\Downarrow}{#1}}

\newcommand{\mult}[2]{{#1}\cdot{#2}}

\newcommand{\cmn}{(\because)}
\newcommand{\mnd}{(\therefore)}
\newcommand{\unt}{\iota}
\newcommand{\cun}{\mathbin{\rotatebox[origin=c]{180}{$\iota$}}}
\newcommand{\dgnl}{\delta}
\newcommand{\dproj}{\,!}
\newcommand{\cgnl}{\varrho}
\newcommand{\cproj}{\mathbin{\rotatebox[origin=c]{180}{!}}}

\newcommand{\spmnd}{\nabla}
\newcommand{\spunt}{\bot}
\newcommand{\spcmn}{{\rm \Delta}}
\newcommand{\spcun}{\top}

\newcommand{\Lr}{{\triangleleft r}}
\newcommand{\Rr}{{r\triangleright}}
\newcommand{\Ll}{{\triangleleft \ell}}
\newcommand{\Rl}{{\ell \triangleright}}

\newcommand{\Dashv}{\mathrel{\mbox{$={\kern-1.1ex}|$}}}
\renewcommand{\vDash}{\mathrel{\mbox{$|{\kern-1.1ex}=$}}}

\newcommand{\op}[1]{{#1}^o}

\newcommand{\EQLS}{$=$}
\newcommand{\IFF}{$\iff$}
\newcommand{\SUB}{$\subseteq$}
\newcommand{\SUP}{$\supseteq$}
\newcommand{\THEN}{$\Longrightarrow$}
\newcommand{\IF}{$\Longleftarrow$}

\begin{document}

\title{\begin{minipage}{1\linewidth}\centering Lambek pregroups are Frobenius spiders
\\ in preorders
\end{minipage}}
\author{Dusko Pavlovic}
\email{dusko@hawaii.edu}
\homepage{dusko.org}
\orcid{0000-0002-9855-6861}
\thanks{Supported by NSF and AFOSR.}
\affiliation{University of Hawaii, Honolulu HI, USA}
\maketitle

\begin{abstract}
\noindent \emph{"Spider"}\/ is a nickname of \emph{special Frobenius algebras}, a fundamental structure from mathematics, physics, and computer science. \emph{Pregroups}\/ are a fundamental structure from linguistics. Pregroups and spiders have been used together in natural language processing: one for syntax, the other for semantics. It turns out that pregroups themselves can be characterized as pointed spiders in the category of preordered relations, where they naturally arise from grammars. The other way around, preordered spider algebras in general can be characterized as unions of pregroups. This extends the characterization of relational spider algebras as disjoint unions of groups. The compositional framework that emerged with the results suggests new ways to understand and apply the basis structures in machine learning and data analysis.
%
%
%
\end{abstract}

\section{Introduction}\label{Sec:intro}

\subsection{Background} Special Frobenius algebras, lovingly nicknamed \emph{spiders}\/ due to their characteristic normal form \eqref{eq:spider}, are frequently encountered on a wide swath of categorical applications, from topological quantum field theories, through group representations and linear algebra, to computability and complexity \cite{PavlovicD:MSCS13,KockJ:frobenius,PavlovicD:IC12,PavlovicD:MonCom,PavlovicD:MonCom3,StreetR:Frobenius}. Pregroups are, on the other hand, the algebraic structure underlying modern categorial grammars \cite{Lambek-Casadio:four,LambekJ:pregroups,LambekJ:grammar97,PrellerA:linear}. Pregroups and spiders have been used together in computational linguistics, one for presenting the syntax, the other for assigning the vector space semantics \cite{Mehrnoosh-Moortgat:parasitic,Mehrnoosh:word-one,Mehrnoosh:word-two}. 
The string diagram notation has been instrumental in combining these two structures of disjoint origins, at different levels of language modeling, to open an alley towards reconciling the two leading paradigms of natural language processing, the distributional and the compositional  \cite{LewisM:transl,Coecke-Martha-Marsden,CoeckeB:discocat,Heunen-Sadrzadeh}.

\smallskip
Lambek's \emph{Switching Lemma} \cite[Prop.~2]{LambekJ:grammar97} says that any sequent in a \emph{free}\/ pregroup can be normalized to a form where all contraction steps precede all expansion steps. It was proved inductively, as a form of cut elimination supported by the free pregroup sequents \cite[Ch.~27]{LambekJ:word-sentence}. The name \emph{"Switching Lemma"}\/ was introduced in Buszkowski's expositions of his decision procedure for the pregroup sequent calculus \cite{BuszkowskiW:synthese07}, and of the equivalence of pregroups and context-free languages \cite{Buszkowski-Moroz:CFL}. The general  method of cut elimination was introduced in Gentzen's seminal paper \cite{GentzenG:seq}, and remained one of the central concepts of sequent calculi.

\subsection{Results} Thm.~\ref{thm:one} says that pregroups and pointed spiders are different presentations of the same variety of sequent algebras. It follows that all pregroups, not only the free ones, reduce to the normal form where all contractions are performed before all expansions. The property does not depend on inductive switching or cut elimination, which only arise in free sequent algebras. While sequent calculi generate free sequent algebras just like term calculi generate the equational algebras, the algebraic reductions, often conveniently presented as categorical diagram chasing, apply to both families of general algebras, not just the free ones. Thm.~\ref{thm:two} is also a statement about general sequent algebras. It says that the residuation monoids are characterized by the Frobenius conditions, which determine the shape of a spider. Note that formal grammars are defined as a special family of sequent calculi (those where the generators are partitioned into the terminals and the nonterminals \cite{ChomskyN:three}), and that the categorial grammars are the sequent calculi corresponding to the residuation monoids \cite[\S 1.2]{Bar-HillelY:categorial,LambekJ:math-sentence,LambekJ:rings}. Thm.~\ref{thm:three} is a statement about the spider algebras in general, which are more general than those in Thm.~\ref{thm:one} because they are not necessarily pointed, but less general than the Frobenius algebras in Thm.~\ref{thm:two}, because the spider algebras are required to satisfy the \emph{special}\/ isometry condition, in addition to the Frobenius conditions. Thm.~\ref{thm:three} says that a general spider can be decomposed into a consistent union of pregroups; and that the union of any consistent family of pregroups yields a spider.  

\smallskip
The upshot is that the spiders do not appear only as the orthonormal bases of the vector space models used to assign meaning to pregroup sentences but also as the shapes of pregroups themselves. Any spider of preordered relations is a union of pregroups. A pregroup is precisely a pointed spider.

\subsection{Related work} The characterization of the relational spiders over preorders extends the characterization \cite{PavlovicD:QI09} of the relational spiders over sets:
\bear
\frac{\mbox{unions of pregroups}}{\mbox{spiders of preordered relations}} & = & \frac{\mbox{disjoint unions of groups}}{\mbox{spiders of relations}} 
\eear
While \cite{PavlovicD:QI09} was mainly concerned with the group structure of the basis elements as a quantum-algorithmic resource \cite{PavlovicD:Qabs12}, pregroups are a structure used in natural language processing \cite{CoeckeB:discocat,CoeckeB:QNLP21,CoeckeB:QNLP20}. The structure of the basis elements and the basis sets correspond, respectively, to the syntactic and the semantical aspects. Modern applications require seamless integration of both aspects \cite{JurafskyD:book-2020}. The presented results have been directly inspired by the DisCoCat program  \cite{CoeckeB:discocat,Heunen-Sadrzadeh,Mehrnoosh:word-one,Mehrnoosh:word-two}, and indirectly by the remarkable headways in the  practices of computational linguistics over theory \cite{GPT3}. The author is at an early stage of studying the connections between the many fronts of progress in computational linguistics  \cite{LewisM:NN,MoortgatM:Dyck,Mehrnoosh-Moortgat:parasitic,PrellerA:linear}.  Beyond the current horizon, the conceptual veins from spiders' roots in categorical quantum mechanics \cite{PavlovicD:QMWS,PavlovicD:Qabs12} to their branchings through the theories of language and communication \cite{Heunen-Sadrzadeh} continue to feed not only the theoretical explorations of sequent algebra \cite{PavlovicD:FunSem17,BonchiF:LICS18,SobocinskiP:POPL21} but also the practical applications and commercial technology transfers  \cite{CoeckeB:QNLP21,CoeckeB:QNLP20}.

\subsection{Why sequent algebra?}
To relate Lambek's pregroups and Frobenius' algebras, both structures are presented in terms of preordered relations. This framework is slightly more general than the standard presentation, and it could be viewed as a novel contribution, or as an unnecessary distraction. In the following, I argue that it is neither.

\smallskip
Equational theories are usually presented as pairs $(\OOO, \EEE)$, where $\OOO$ is a set of operations given with arities, and $\EEE$ is a set of equations between some terms generated from $\OOO$. Algebra is the practice of deriving other valid equations. If the operations are presented as arrows, the composite operations are the paths along the arrows, the  equations are the faces of directed graphs built from the arrows, and the sequences of equations can be presented as \emph{diagram chases}, going back to \cite{LambekJ:rings}. 

\smallskip

Sequent theories are also specified as pairs $(\OOO, \EEE)$, but the operations from $\OOO$ are now given as sequents, and $\EEE$ is a set of implications between some sequents derived from $\OOO$. Since sequents are bulkier than  terms, the sequences of sequent implications are usually written as proof trees. This gives the sequent calculi their typical appearance. Since the present work only gives rise to relatively simple sequents, I will not display the proof trees, but save space by writing the sequent derivations horizontally, with the sequents usually enclosed in parentheses. This should not conceal the fact that we are treading on the well ploughed ground of sequent algebra, going back to Axel Thue \cite{ThueA:selected}. The category $\PRel$ is just a convenient categorical framework for it. %
%
%
%
But the use of sequents goes well beyond algebra. The main techniques were developed by Gerhard Gentzen, Emil Post and Andrey Markov\footnote{The same one who invented Markov chains \emph{and}\/ the $n$-gram model of language, all in the same paper \cite{MarkovA:onegin}.}. Noam Chomsky refined and specialized sequent theories into grammars \cite{ChomskyN:three,ChomskyN:structures}. Formal grammars have been instrumental in specifying the  syntax of programming languages \cite{NaurP:algol}, as well as in their operational semantics \cite{PlotkinG:SOS}. The use of sequents for specifying grammars of the \emph{natural}\/ languages goes back to Sanskrit philologist Panini from 6th century~BC \cite{BloomfieldL:Panini}. Lambek's pregroups belong to that tradition. The results show that the sequent view of the spider algebras and of the Frobenius conditions is a natural extension. 

\subsection{Overview of the paper} In Sec.~\ref{Sec:prel} we set the stage for sequent algebra in the category of preorders and preordered relations, shortened to \emph{prelations}. A reader with some categorical experience will be familiar with most of the concepts introduced in this section, and may want to skip forward and come back as needed. Sec.~\ref{Sec:preg} describes the prelational monoids, one of the simplest structures of sequent algebra, and characterizes when they are representable as the ordinary preordered monoids. The structure of pregroups is also presented in this framework. Sec.~\ref{Sec:frobspid} spells out the Frobenius conditions and the  special isometry condition for prelations, and defines the spider algebras in this framework. Theorems \ref{thm:one}, \ref{thm:two} and \ref{thm:three} are stated and proved in sections \ref{Sec:thm},
 \ref{Sec:residuated}, and \ref{Sec:unions}, respectively. Sec.~\ref{Sec:outro} discusses the repercussions of the results, in particular with respect to the goals and the limitations of pregroup grammars \cite[Ch.~28]{LambekJ:word-sentence}.

%
%

\section{Preorders and prelations}\label{Sec:prel}

A \emph{preorder}\/ is a set $A$ with a transitive, reflexive relation $\der A \subseteq  A\times A$, which means that
\beq \derrr x A x \qquad \qquad \mbox{and} \qquad \qquad \derrr x  A y\  \wedge\  \derrr y A z\ \  \implies\ \  \derrr x  A z \eeq
hold for all $x,y,z\in A$. When the underlying set $A$ is clear or irrelevant, we write $\derr x y$. The \emph{preorder equivalence}\/ is the symmetric part of a preorder, defined by
\bea\label{eq:eq}
\derreq x y & \iff & \derr x y\wedge\derr y x
\eea
The quotient of $A$ modulo the equivalence relation $\dereq A$ is the largest poset (partially ordered set) with a monotone map from $A$. Since this map is a surjective order embedding, every preorder is equivalent\footnote{Order embeddings are a special case of what category theorists call full and faithful functors. Surjective order embeddings are thus a special case of categorical equivalences.} to a poset.

\smallskip
A \emph{prelation}\/ $\Phi:A\prel B$ between preorders $A$ and $B$ is a relation $\der \Phi \subseteq A\times B$ which is upper-closed in $A$ and lower-closed in $B$, i.e.
\bea\label{eq:prel}
\derrr x A {x'}\ \wedge\  \derrr{x'} \Phi {y'}\ \wedge\  \derrr {y'}  B y &\implies & \derrr x \Phi y
\eea
holds for all $x,x' \in A$ and $y,y' \in B$. Given a prelation $\Psi :B\prel  C$, the composite $\scomp \Phi \Psi \colon A\prel  C$ is defined
\bea
\derrr x {(\Phi\, ; \Psi)} z & \iff & \exists y.\ \derrr x \Phi y\  \wedge\ \derrr y\Psi z
\eea
The transitivity makes the preordering $\der A \subseteq A\times A$ into the identity prelation on $A$. The category of preorders and prelations is denoted by $\PRel$. 

\paragraph{\sf Examples.} Given a set $A$, any set of sequents $\{{a_{i}} \leftarrow{a'_{i}}\}_{i\in J}\subseteq A\times A$ determines a preorder on $A$ as its transitive reflexive closure. Given preorders $A$ and $B$, any set of sequents $\{{a_{i}} \leftarrow{b_{i}}\}_{i\in K}\subseteq A\times B$ determines a prelation $A\prel B$ as its transitive closure under $\der A$ on the left and $\der B$ on the right.

\subsection{Monoidal structure}\label{Sec:Mon}
The main tensor products of $\PRel$ are familiar from their restrictions to $\PRel$'s smaller cousin $\Rel$, the category of sets and relations. For any pair of preorders $A$ and $B$, 
\begin{itemize}
\item the disjoint union $A+B$ with the preorder
\bea
\derrr x{A+B} y & \iff & \derrr x A y \vee \derrr x B y
\eea
is a biproduct in $\PRel$, with the empty set $\emptyset$ as the zero;
\item the cartesian product $A\times B$ with the preorder written in the form
\bea
\derrr{x, y} {AB} {x', y'} & \iff& \derrr x A {x'}\wedge \derrr y B {y'}
\eea
is a symmetric monoidal product in $\PRel$, with the singleton $\OOne=\{\emptyset\}$ as the unit object. 
\end{itemize}
The diagonals $\dgnl:A\prel  A\times A$ and the projections $\dproj :A\prel  \OOne$ are defined by
\beq\label{eq:delta}
\derrr x \dgnl {y,z} \ \iff \ \derrr {x,x} {AA} {y,z} 
\qquad \qquad \mbox{and}\qquad\qquad \derrr x \dproj \emptyset \mbox{ for all } x
\eeq
Just like in the category $\Rel$ of sets and relations, the diagonals and the projections do not form natural transformations: the former commutes only with the single-valued prelations, the latter only with the total ones. $(\PRel, \times , \OOne)$ is therefore not a cartesian structure, just symmetric monoidal. $(\PRel, +, \emptyset)$ is a bicartesian structure.

\subsection{Completions} \label{Sec:Completions}
Any preorder $X$ embeds into the complete lattices of its lower sets and upper sets, respectively, with the opposite inclusion orderings:
\begin{align}
\Do X & =  \{ L\subseteq X\ |\ x\in L\wedge \derrr x X {x'} \implies x'\in L\} && \mbox{ with }& \derrr L{\Do X} {L'} & \iff L \supseteq L' \\
\Up X & =  \{ V\subseteq X\ |\ \derrr x X  {x'} \wedge x'\in V \implies x \in V\}  &&\mbox{ with }& \derrr U{\Up X} {U'} & \iff U \subseteq U'
\end{align}
The embeddings map each $x\in X$ into the principal lower set $\downarrow x = \{y|\derrr x X {y}\}\in \Do X$ and the principal upper set $\uparrow x = \{y|\derrr y X x\}\in \Up X$. It is easy to see that $\Do X$ is generated by the unions of the principal lower sets as the suprema, and that $\Up X$ is generated by the unions of the principal upper sets as the infima. They are thus $X$'s supremum and infimum completions, respectively.

\subsection{Dualities}\label{Sec:duality}
$\PRel$ supports two dualities:
\bea
O, \ddag \ \colon\ \PRel^{op} & \to & \PRel
\eea
where $\PRel^{op}$ is the opposite category of $\PRel$, i.e. $\PRel^{op}(A, B) = \PRel(B,A)$. It is convenient to write $A^o$ instead of $O(A)$ and $A^\ddag$ instead of $\ddag(X)$.  The object part of $O$ sends $A$ to the opposite preorder $A^o$, i.e. $\derrr x {A^o} {x'} \iff  \derrr {x'} A x$. The object part of $\ddag$ is the identity, i.e. $A^\ddag = A$. The arrow parts map a prelation $\Phi :A \prel  B$ to prelations in the form $\Phi^o : B^o \prel  A^o$ and $\Phi^\ddag: B\prel  A$, defined
\beq\label{eq:duals}
\derrr y{\Phi^o} x \iff \derrr x\Phi y \qquad\qquad \derrr y {\Phi^\ddag} x \iff \bigg(\forall uv.\derrr u \Phi v \Rightarrow \derrr y B v \wedge \derrr u A x  \bigg)
\eeq
To understand $\ddag \ \colon\ \PRel^{op} \to \PRel$, consider the special case when $A = \OOne$. A  prelation $\Phi:\OOne \prel B$ can then be viewed as a lower set $\Phi\in \Do B$, and it is easy to see that $\Phi^\ddag: B\prel \OOne$, defined as in \eqref{eq:duals}, is the upper set $\Phi^\ddag \in \Up B$ of all upper bounds of the lower set $\Phi$. When $B=\OOne$, then $\Psi : A\prel \OOne$ is an upper set $\Psi \in \Up A$ and $\Psi^\ddag:\OOne \prel A$ in \eqref{eq:duals} corresponds to the lower set $\Psi^\ddag \in \Do A$ of $\Psi$'s lower bounds. Iterating the dagger induces the closure operators $\ddag\ddag: \Do B \to \Do B$ where $\Phi = \Phi^{\ddag\ddag}$ if and only if the lower set $\Phi$ contains all lower bounds of the set of its upper bounds. The closure operators $\ddag\ddag: \Up A \to \Up A$ fix the upper sets $\Psi = \Psi^{\ddag\ddag}$ that contain all upper bounds of the sets of their lower bounds. If a general prelation $\Phi:A\prel B$ is viewed as a lower set $\Phi\in \Do(A^o \times B)$, then $\Phi^\ddag$ is the set of $\Phi$'s upper bounds. A set of upper bounds is, of course, an element of $\Up(A^o\times B)$, but if it is viewed as the lower set $\Phi^\ddag \in \Do(B^o\times A)\cong \Up(A^o\times B)$, then it easier to see that it is a prelation $\Phi^\ddag:B\prel A$. Checking that $\Phi\subseteq \Phi^{\ddag\ddag}$ and $\Phi^\ddag = \Phi^{\ddag\ddag\ddag}$ always hold is routine, and $\Phi = \Phi^{oo}$ is obvious. The two dualities coincide just on the preorders that happen to be equivalence relations.  The duality $O$ gives rise to the internal adjoints in $\PRel$ on the level of objects, $\ddag$ on the level of arrows. The former adjoints make $\PRel$ into a compact category, whereas the latter provide an internal characterization of the monotone \emph{maps}\/ between preorders.

\subsection{Compact structure}
On the level of the objects of $\PRel$, the fact that every preorder $B$ has a right adjoint\footnote{This adjunction is with respect to the view of the monoidal category $\PRel$ as a bicategory with a single 0-cell, which makes it its objects into 1-cells, and its morphisms into 2-cells.} $\op B$ in $\PRel$ is realized by the bijections
\bea
\PRel(A\times B, C) & \cong & \PRel(A, \op B \times C)
\eea
natural in $A$ and $C$. The adjunction unit $ \OOne\tto{\, \eta\, } \op B \times B$ arises on the right, corresponding to $\id:B\prel  B$ on left, whereas the counit $B\times \op B\tto{\,  \varepsilon\, } \OOne$ arises on the left corresponding to $\id:B\prel  B=B^{oo}$ on the right. This makes $\PRel$ into a compact category. 

\subsection{Maps}\label{Sec:maps}
On the level of morphisms, the adjunction conditions  are
\begin{multline}\label{eq:tot}
\id_A \subseteq \scomp \Phi{\Phi^\ddag}, \mbox{ which means}\\ \derrr x A {x'} \ \implies \ \left[\exists y.\ \derrr x\Phi y\  \wedge\ \bigg(\forall uv.\derrr u \Phi v \Rightarrow \derrr y B v \wedge \derrr u A {x'}\bigg)\right]
\end{multline}
\begin{multline}
\scomp{\Phi^\ddag}\Phi  \subseteq \id_B, \mbox{ which means}\\ \left[\exists x.\ \bigg(\forall uv.\derrr u \Phi v \Rightarrow \derrr y B v \wedge \derrr u A x\bigg) \wedge \derrr x\Phi {y'}\right]\ \implies \ \derrr y B {y'}\label{eq:sv}
\end{multline}
A prelation $\Phi$  satisfying (\ref{eq:tot}--\ref{eq:sv}) is called a \emph{map}. Condition \eqref{eq:tot} is equivalent to the claim that for every $x\in A$ there is some $y\in B$ with $\derrr x \Phi y$; condition \eqref{eq:sv} to the claim that $\derrr x \Phi y$ and $\derrr x \Phi {y'}$ together imply $\derrreq y B {y'}$. In other words, \eqref{eq:tot} says that $\Phi$ is total, \eqref{eq:sv} that it is single-valued, up to the preorder equivalence. They hold together if and only if for every $a\in A$ there is $\Phi(a)\in B$, unique up to the preorder equivalence, such that
\bea\label{eq:repres}
\derrr a \Phi b & \iff & \derrr {\Phi(a)} B b
\eea
Condition \eqref{eq:repres} means that the lower set of $B$-elements $y$ that are $\Phi$-related to $a$ is \emph{representable} by $\Phi(a)$. It is easy to see that any $y\in B$ also satisfies \eqref{eq:repres} if and only if  $\derrreq {y} B {\Phi(a)}$. When $B$ is a partial order, the preorder equivalence boils down to the equality, and a map $\Phi:A\prel B$ boils down to a monotone function $\Phi:A\to B$ with \eqref{eq:repres}. Maps between general preorders are still equivalent to monotone functions, but the representation is up to the preorder equivalence, and the representing functions need to be chosen. 

\subsection{Representations over completions} 
By \eqref{eq:prel} a prelation $\Phi:A\prel B$ can be viewed as either of the functions
\begin{align*}
\Phi_\bullet \colon A & \to \Do B & \Phi^\bullet \colon B & \to \Up A\\
a & \mapsto \left\{y|\derrr a \Phi y\right\} & b & \mapsto \left\{x|\derrr x \Phi b\right\}
\end{align*}
and extended to 
\begin{align*}
\Phi_\ast \colon \Up A & \to   \Do B & \Phi^\ast \colon \Do B & \to \Up A\\
V & \mapsto  \bigcap_{a\in V} \Phi_\bullet(a) 
& 
L & \mapsto \bigcap_{b\in L} \Phi^\bullet(b) 
\end{align*}
which form a Galois connection
\bea\label{eq:galoisast}
\Phi^\ast(L) \supseteq V & \iff & L\subseteq \Phi_\ast(V)
\eea
Since $\Up B^o = \Do B$ and $\Do A^o = \Up A$, the extensions of the dual $\Phi^o:B^o\prel A^o$ are simply $\Phi^o_\ast = \Phi^\ast$ and $\Phi^{o\ast} = \Phi_\ast$. On the other hand, the extensions of  the dual $\Phi^\ddag:B\prel A$ are in the form
\begin{align*}
\Phi_\# \colon \Up B & \to   \Do A & \Phi^\# \colon \Do A & \to \Up B\\
W & \mapsto  \bigcap_{b\in W} \left\{x|\derrr b {\Phi^\ddag} x\right\}&   
K & \mapsto   \bigcap_{a\in K} \left\{y|\derrr y {\Phi^\ddag} a\right\}
\end{align*}
which again form a Galois connection
\bea\label{eq:galoisdual}
\Phi^\#(K) \supseteq W & \iff & K\subseteq \Phi_\#(W)
\eea
The composites of any pair of maps that form a Galois connection are closure operators, and the induced lattices of closed sets carry a reduced Galois connection. This is the \emph{nucleus}\/ of the original Galois connection \cite{PavlovicD:ICFCA12,PavlovicD:Samson13,PavlovicD:nucleus,WillertonS:nucleus}. A Galois connection is \emph{nuclear}\/ when it is its own nucleus. For the Galois connections in \eqref{eq:galoisast} and \eqref{eq:galoisdual}, the lattices of closed sets are isomorphic. A prelation satisfies $\Phi = \Phi^{\ddag\ddag}$ precisely when the induced Galois connections are nuclear. A representable prelation, i.e. a monotone map $\Phi: A\to B$ is always nuclear, i.e. satisfies $\Phi = \Phi^{\ddag\ddag}$. This follows from the observation that the extensions $\Phi_\ast$ and $\Phi^\#$ factor through $\Phi$ as its extensions 
\beq
\Phi_\ast(V)\ =\ \bigcap_{a\in V} \downarrow\! \Phi(a)\qquad \qquad \qquad \qquad \Phi^\#(K)\ =\ \bigcap_{a\in K} \uparrow\! \Phi(a)
\eeq
where $\downarrow\colon B\to \Do B$ and $\uparrow\colon B\to \Up B$ are the supremum and the infimum completion embeddings, respectively, mapping $b\in B$ to the lower set $\downarrow\! b = \{y|\derrr b B y\}$ and to the upper set $\uparrow\! b = \{y|\derrr y B b\}$. For a map  $\Phi: A\to B$, the representation from \eqref{eq:repres} thus extends to
\bea\label{eq:represdag}
\derrr b {\Phi^\ddag} a & \iff & \derrr b B {\Phi(a)}
\eea

\subsection{Diagonals and projections are maps} The diagonal prelation $\dgnl:A\prel  A\times A$ and the projection $\dproj :A\prel  \OOne$ are represented by $\dgnl(x) = <x,x> \in A\times A$ and $\dproj(x) = \emptyset$. Instantiating \eqref{eq:duals}, their duals  $\dgnl^\ddag : A \times A \prel  A$ and $\dproj^\ddag: \OOne \prel  \op A$  are defined
\bear
\derrr {x,y} {\dgnl^\ddag} z  & \iff & \bigg(\forall uvw. \derrr {u,u} {AA} {v,w}\implies \derrr{ x, y, u}{AAA}{v,w,z}\bigg)\hspace{3em}  \mbox{and}\\
\derrr \emptyset {\dproj^\ddag} x & \iff & \bigg(\forall u. \derrr u \dproj \emptyset \implies \derrr u A x\bigg)
\eear  
Simplifying, and writing $\cgnl = \dgnl^\ddag$ and $\cproj = \dproj^\ddag$ we have 
\beq
\derrr {x,y} {\cgnl} z  \  \iff \  \derrr {x, y} {AA} {z,z}\qquad\qquad \mbox{ and } \qquad \qquad 
\derrr \emptyset {\cproj} x \ \iff  \bigg(\forall u. \derrr u A x\bigg)
\eeq  
It is now easy to prove
\begin{align}
\id_A  &\subseteq \scomp {\dgnl} {\cgnl} & \scomp{\cgnl}{\dgnl}  & \subseteq \id_{AA} \\
\id_A &\subseteq \scomp \dproj{\cproj} & \scomp{\cproj}\dproj  & \subseteq \id_{\OOne}
\end{align}
which means that $\dgnl$ and $\dproj$ are maps. In string diagrams, they are usually drawn as black-dot comonoids and monoids:
\beq
\begin{split}
\newcommand{\diagonal}{\dgnl}
\newcommand{\projn}{\dproj}
\newcommand{\codiagonal}{\cgnl}
\newcommand{\coprojn}{\cproj}
\def\JPicScale{.9}
\ifx\JPicScale\undefined\def\JPicScale{1}\fi
\psset{unit=\JPicScale mm}
\psset{linewidth=0.3,dotsep=1,hatchwidth=0.3,hatchsep=1.5,shadowsize=1,dimen=middle}
\psset{dotsize=0.7 2.5,dotscale=1 1,fillcolor=black}
\psset{arrowsize=1 2,arrowlength=1,arrowinset=0.25,tbarsize=0.7 5,bracketlength=0.15,rbracketlength=0.15}
\begin{pspicture}(0,0)(81.09,15)
\psline(0,15)(0,12.5)
\psline(10,15)(10,12.5)
\psline(0,12.5)(5,7.5)
\psline(5,7.5)(10,12.5)
\psline(5,0)(5,7.5)
\psline(60,2.5)(60,0)
\psline(70,2.5)(70,0)
\psline(65,7.5)(70,2.5)
\psline(60,2.5)(65,7.5)
\psline(65,7.5)(65,15)
\psline(80,7.5)(80,12.5)
\rput{0}(18.75,7.5){\psellipse[fillstyle=solid](0,0)(1.09,-1.09)}
\psline(18.75,0)(18.75,7.5)
\rput(5,12.5){$\diagonal$}
\rput(65,1.25){$\codiagonal$}
\rput(18.75,12.5){$\projn$}
\rput(80,1.25){$\coprojn$}
\rput{0}(65,7.5){\psellipse[fillstyle=solid](0,0)(1.09,-1.09)}
\rput{0}(80,7.5){\psellipse[fillstyle=solid](0,0)(1.09,-1.09)}
\rput{0}(5,7.5){\psellipse[fillstyle=solid](0,0)(1.09,-1.09)}
\end{pspicture}

\end{split}
\eeq
It is interesting to note that $\cgnl: A\times  A\prel A$ is a map if and only if the preorder $A$ has finite nonempty infima, whereas $\dgnl^o : A^o\times A^o \prel A^o$ is a map if and only if $A$ has finite nonempty suprema. The prelation $\cproj \colon \OOne\prel A$ is a map just when $A$ has a bottom element. 

%

\section{Pregroups}\label{Sec:preg}

\subsection{Preordered monoids}\label{Sec:mon}
A monoid in $\PRel$ is a triple $(A,\spmnd,\spunt)$ where $A$ is a preorder, $\spmnd:A\times A\prel A$ a ternary prelation, i.e. a subset of $A\times A\times A$ which is upper-closed in the first two components and lower-closed in the third one, $\spunt:\OOne \prel A$ is a lower set in $A$. They are required to satisfy the usual monoid axioms, the associativity and the unit conditions, which correspond in $\PRel$ to the following sequents:
 
 \medskip
\beq\label{eq:monoid}
\begin{split}
\newcommand{\eqone}{\scriptstyle \exists u.\derrr {x,y}\spmnd u \wedge \derrr {u, z} \spmnd w} 
\newcommand{\eqtwo}{\scriptstyle \exists v. \derrr {x,v}\spmnd w\wedge \derrr{y, z} \spmnd v}
\newcommand{\eqthree}{\scriptstyle \derrr {\spunt,x} \spmnd y}
\newcommand{\eqfour}{\scriptstyle \derrr{x,\spunt} \spmnd y}
\newcommand{\eqfive}{\scriptstyle \derrr x A y}
\newcommand{\eex}{\scriptstyle x}
\newcommand{\eey}{\scriptstyle y}
\newcommand{\eez}{\scriptstyle z}
\newcommand{\eeu}{\scriptstyle u}
\newcommand{\eev}{\scriptstyle v}
\newcommand{\eew}{\scriptstyle w}
\newcommand{\pregr}{\scriptstyle \spmnd}
\newcommand{\prunt}{\scriptstyle \spunt}
\def\JPicScale{1.079}
\ifx\JPicScale\undefined\def\JPicScale{1}\fi
\psset{unit=\JPicScale mm}
\psset{linewidth=0.3,dotsep=1,hatchwidth=0.3,hatchsep=1.5,shadowsize=1,dimen=middle}
\psset{dotsize=0.7 2.5,dotscale=1 1,fillcolor=black}
\psset{arrowsize=1 2,arrowlength=1,arrowinset=0.25,tbarsize=0.7 5,bracketlength=0.15,rbracketlength=0.15}
\begin{pspicture}(0,0)(121.25,27.5)
\rput[r](25.63,27.5){$\eqone$}
\rput[l](31.87,27.5){$\eqtwo$}
\rput(28.75,10){\EQLS}
\psline(0,5)(0,0)
\psline(10,5)(10,0)
\psline(15,10)(15,0)
\psline(6.25,8.75)(10,5)
\psline(6.25,11.25)(8.75,13.75)
\psline(11.25,13.75)(15,10)
\psline(10,20)(10,16.25)
\psline(60.63,5)(60.63,0)
\psline(50.63,5)(50.63,0)
\psline(45.63,10)(45.62,0)
\psline(54.38,8.75)(50.63,5)
\psline(45.63,10)(49.38,13.75)
\psline(51.88,13.75)(54.38,11.25)
\psline(50.63,20)(50.62,16.25)
\rput(91.25,10){\EQLS}
\psline(86.25,10)(86.25,0)
\psline(77.5,11.25)(80,13.75)
\psline(82.5,13.75)(86.25,10)
\psline(81.25,20)(81.25,16.25)
\psline(120,20)(120,0)
\psline(96.25,10)(96.25,0)
\psline(96.25,10)(100,13.75)
\psline(102.5,13.75)(105,11.25)
\psline(101.25,20)(101.25,16.25)
\rput(111.25,10){\EQLS}
\rput[r](88.13,27.5){$\eqthree$}
\rput(101.25,27.5){$\eqfour$}
\rput[l](114.37,27.5){$\eqfive$}
\rput(91.25,27.5){\IFF}
\rput(28.75,27.5){\IFF}
\rput(111.25,27.5){\IFF}
\rput[r](-1.25,0){$\eex$}
\rput[r](8.75,0){$\eey$}
\rput[r](13.75,0){$\eez$}
\rput[br](6.87,13.12){$\eeu$}
\rput[bl](53.75,13.12){$\eev$}
\rput[l](51.87,19.38){$\eew$}
\rput[l](11.25,19.38){$\eew$}
\rput[r](44.38,0){$\eex$}
\rput[r](85,0){$\eex$}
\rput[r](95,0){$\eex$}
\rput[r](118.75,0){$\eex$}
\rput[r](49.37,0){$\eey$}
\rput[l](82.5,20){$\eey$}
\rput[l](102.5,20){$\eey$}
\rput[l](121.25,20){$\eey$}
\rput[r](59.37,0){$\eez$}
\psline(56.88,8.75)(60.62,5)
\psline(0,5)(3.75,8.75)
\psline(8.75,16.25)(11.25,16.25)
\psline(8.75,13.75)(11.25,13.75)
\psline(11.25,16.25)(11.25,13.75)
\psline(8.75,16.25)(8.75,13.75)
\psline(49.38,16.25)(51.88,16.25)
\psline(49.38,13.75)(51.88,13.75)
\psline(51.88,16.25)(51.88,13.75)
\psline(49.38,16.25)(49.38,13.75)
\psline(3.75,11.25)(6.25,11.25)
\psline(3.75,8.75)(6.25,8.75)
\psline(6.25,11.25)(6.25,8.75)
\psline(3.75,11.25)(3.75,8.75)
\psline(54.38,11.25)(56.88,11.25)
\psline(54.38,8.75)(56.88,8.75)
\psline(56.88,11.25)(56.88,8.75)
\psline(54.38,11.25)(54.38,8.75)
\psline(75,11.25)(77.5,11.25)
\psline(75,8.75)(77.5,8.75)
\psline(77.5,11.25)(77.5,8.75)
\psline(75,11.25)(75,8.75)
\psline(80,16.25)(82.5,16.25)
\psline(80,13.75)(82.5,13.75)
\psline(82.5,16.25)(82.5,13.75)
\psline(80,16.25)(80,13.75)
\psline(100,16.25)(102.5,16.25)
\psline(100,13.75)(102.5,13.75)
\psline(102.5,16.25)(102.5,13.75)
\psline(100,16.25)(100,13.75)
\psline(105,11.25)(107.5,11.25)
\psline(105,8.75)(107.5,8.75)
\psline(107.5,11.25)(107.5,8.75)
\psline(105,11.25)(105,8.75)
\rput(10,15){$\pregr$}
\rput(5,10){$\pregr$}
\rput(50.62,15){$\pregr$}
\rput(55.62,10){$\pregr$}
\rput(81.25,15){$\pregr$}
\rput(101.25,15){$\pregr$}
\rput(76.25,10){$\prunt$}
\rput(106.25,10){$\prunt$}
\end{pspicture}

\end{split}
\eeq

\paragraph{\sf Examples.} Let $\MMM$ be the set of finite multisets of natural numbers, which can be viewed as functions $\NNn\to \NNn$ that are $0$ everywhere except at a finite number of points. Let the preordering be
\bear
\derrr x A y & \iff & \Sigma x \leq \Sigma y
\eear
where $\Sigma x$ abbreviates the number $\Sigma_{n\in \NNn}nx_{n}$. A prelational monoid can then be defined by
\bear
\derrr {x, y} \spmnd z & \iff &  \derrr {x+y} A z
\eear
where $x+y$ is the multiset of $(x+y)_{n} = x_{n}+y_{n}$. This prelational monoid is thus represented by the operation $+$ on $\MMM$. For a non-representable an  non-commutative example, the reader may want to spell out the monoid over the set $\LLL$ of all words in some alphabet preordered by the prelation $\derr x y$ saying that the word $x$ occurs within the word $y$, not necessarily as a prefix or contiguously. An interesting prelational monoid then admits $\derrr{x,y}\spmnd z$ whenever $z$ contains some shuffle of $x$ and $y$. This is obviously not commutative and  in the next section we explain what does it mean that it is not representable.

\subsection{Representable, strict, pointed monoids}\label{Sec:repres}
Lambek's pregroups \cite{LambekJ:pregroups,LambekJ:word-sentence} are monoids over preorders, with additional features and with special requirements. The additional features are the adjunction operations, discussed in Sec.~\ref{Sec:adj}. The first special requirement is that Lambek defined pregroups as \emph{partially ordered}\/ monoids, and not just preordered as above. A preordering is a partial ordering when $\derr x y$ and $\derr y x$ imply $x=y$. Since any preorder is order-isomorphic with its quotient\footnote{In some situations, splitting the quotient map to embed the quotient back into the original preorder requires the axiom of choice. Such situations do not arise in the present paper.} poset along the preorder equivalence \eqref{eq:eq}, the choice between posets and preorders is a matter of notational conveniences. Posets are often more concise, but maintaining the antisymmetry in sequent algebra requires factoring out sequent loops at every step, so we leave pregroups preordered. 

\smallskip 
The second special requirement is that pregroups are \emph{representable}\/ as monoids. In other words, the underlying monoid structure $(A,\mnd,\unt)$ of a pregroup in $\PRel$ is represented by monotone maps, the prelations that are total and single-valued, as defined in Sec.~\ref{Sec:maps}. A monoid $(A,\mnd,\unt)$ in $\PRel$ is said to be
\begin{itemize}
\item \emph{strict}\/ if $\mnd$ is a map, i.e.  
\beq\label{eq:monmap}
\begin{split}
\newcommand{\pregr}{\scriptstyle \therefore}
\newcommand{\copregr}{\scriptstyle \because}
\def\JPicScale{1.1}
\ifx\JPicScale\undefined\def\JPicScale{1}\fi
\psset{unit=\JPicScale mm}
\psset{linewidth=0.3,dotsep=1,hatchwidth=0.3,hatchsep=1.5,shadowsize=1,dimen=middle}
\psset{dotsize=0.7 2.5,dotscale=1 1,fillcolor=black}
\psset{arrowsize=1 2,arrowlength=1,arrowinset=0.25,tbarsize=0.7 5,bracketlength=0.15,rbracketlength=0.15}
\begin{pspicture}(0,0)(80,27.5)
\psline(23.75,11.25)(23.75,8.75)
\psline(26.25,8.75)(30,5)
\psline(23.75,8.75)(20,5)
\psline(1.25,27.5)(1.25,0)
\psline(25,16.25)(25,11.25)
\psline(23.75,8.75)(26.25,8.75)
\psline(23.75,11.25)(26.25,11.25)
\psline(26.25,8.75)(26.25,11.25)
\psline(8.75,27.5)(8.75,0)
\psline(20,5)(20,0)
\psline(30,5)(30,0)
\psline(20,22.5)(23.75,18.75)
\psline(30,22.5)(26.25,18.75)
\psline(20,27.5)(20,22.5)
\psline(30,27.5)(30,22.5)
\psline(66.25,21.25)(70,17.5)
\psline(63.75,21.25)(60,17.5)
\psline(65,3.75)(65,0)
\psline(60,17.5)(60,10)
\psline(70,17.5)(70,10)
\psline(60,10)(63.75,6.25)
\psline(70,10)(66.25,6.25)
\psline(65,27.5)(65,23.75)
\psline(80,27.5)(80,0)
\rput(75,13.75){\SUB}
\rput(15,13.75){\SUB}
\psline(23.75,18.75)(23.75,16.25)
\psline(23.75,16.25)(26.25,16.25)
\psline(23.75,18.75)(26.25,18.75)
\psline(26.25,16.25)(26.25,18.75)
\psline(63.75,23.75)(63.75,21.25)
\psline(63.75,21.25)(66.25,21.25)
\psline(63.75,23.75)(66.25,23.75)
\psline(66.25,21.25)(66.25,23.75)
\psline(63.75,6.25)(63.75,3.75)
\psline(63.75,3.75)(66.25,3.75)
\psline(63.75,6.25)(66.25,6.25)
\psline(66.25,3.75)(66.25,6.25)
\rput(25,10){$\pregr$}
\rput(25,17.5){$\copregr$}
\rput(65,5){$\copregr$}
\rput(65,22.5){$\pregr$}
\end{pspicture}

\end{split}
\eeq
\item \emph{pointed}\/ if $\unt$ is a map, i.e.
\beq\label{eq:untmap}
\begin{split}
\newcommand{\pregr}{\scriptstyle \unt}
\newcommand{\copregr}{\scriptstyle \mathbin{\rotatebox[origin=c]{180}{$\scriptstyle \unt$}}}
\newcommand{\zzzz}{\id_\OOne}
\def\JPicScale{1.1}
\ifx\JPicScale\undefined\def\JPicScale{1}\fi
\psset{unit=\JPicScale mm}
\psset{linewidth=0.3,dotsep=1,hatchwidth=0.3,hatchsep=1.5,shadowsize=1,dimen=middle}
\psset{dotsize=0.7 2.5,dotscale=1 1,fillcolor=black}
\psset{arrowsize=1 2,arrowlength=1,arrowinset=0.25,tbarsize=0.7 5,bracketlength=0.15,rbracketlength=0.15}
\begin{pspicture}(0,0)(80,17.5)
\psline(21.25,6.25)(21.25,3.75)
\psline(22.5,11.25)(22.5,6.25)
\psline(21.25,3.75)(23.75,3.75)
\psline(21.25,6.25)(23.75,6.25)
\psline(23.75,3.75)(23.75,6.25)
\psline(68.75,3.75)(68.75,0)
\psline(68.75,17.5)(68.75,13.75)
\psline(80,17.5)(80,0)
\rput(75,8.75){\SUB}
\rput(15,8.75){\SUB}
\psline(21.25,13.75)(21.25,11.25)
\psline(21.25,11.25)(23.75,11.25)
\psline(21.25,13.75)(23.75,13.75)
\psline(23.75,11.25)(23.75,13.75)
\psline(67.5,13.75)(67.5,11.25)
\psline(67.5,11.25)(70,11.25)
\psline(67.5,13.75)(70,13.75)
\psline(70,11.25)(70,13.75)
\psline(67.5,6.25)(67.5,3.75)
\psline(67.5,3.75)(70,3.75)
\psline(67.5,6.25)(70,6.25)
\psline(70,3.75)(70,6.25)
\rput(22.5,5){$\pregr$}
\rput(22.5,12.5){$\copregr$}
\rput(68.75,5){$\copregr$}
\rput(68.75,12.5){$\pregr$}
\rput(8.75,8.75){$\zzzz$}
\end{pspicture}

\end{split}
\eeq
\item \emph{representable}\/ if it is both strict and pointed.
\end{itemize}
The duals $\cmn = \mnd^\ddag$ and $\cun =\unt^{\ddag}$ are drawn as the vertical flips of the diagrams for $\mnd$ and $\unt$. By \eqref{eq:repres} and \eqref{eq:represdag}, the strictness of $\mnd:A\times A\prel A$ implies that for every $x, y\in A$, there is a representative $\mult x y\in A$ such that
\beq\label{eq:represpreg}
\derrr {x,y}\mnd u \  \iff \  \derrr{\mult x y} A u\qquad \qquad \qquad\qquad \derrr u\cmn {x,y} \  \iff \  \derrr u A {\mult x y}
\eeq
By the same token, the unit $\unt:\OOne \prel A$ is pointed when it has a representative, also denoted $\unt\in A$ by abuse of notation, such that 
\beq\label{eq:pointed}
\derrr \emptyset \unt u \  \iff \  \derrr \unt A u\qquad \qquad \qquad\qquad \derrr u\cun \emptyset \  \iff \  \derrr u A \unt
\eeq
In other words, while a general prelational unit corresponds to a lower set $\unt\in \Do A$, a pointed unit is a principal lower set $\downarrow \unt$, and the corresponding counit $\cun \in \Up A$ is the upper set $\uparrow \unt$.
Note that a prelation $\OOne\prel A$ can fail to satisfy the left-hand condition in \eqref{eq:untmap} only if it is empty. A prelational monoid on $A$ can have an empty unit only if $A$ itself is empty. On the other hand, the right-hand requirement for a pointed monoid in \eqref{eq:untmap} implies the right-hand condition for strictness in \eqref{eq:monmap}:
\beq\label{eq:monun}
\begin{split}
\newcommand{\pregr}{\scriptstyle \therefore}
\newcommand{\copregr}{\scriptstyle \because}
\newcommand{\prunt}{\scriptstyle \unt}
\newcommand{\prcun}{\scriptstyle \mathbin{\rotatebox[origin=c]{180}{$\scriptstyle \unt$}}}
\def\JPicScale{1.1}
\ifx\JPicScale\undefined\def\JPicScale{1}\fi
\psset{unit=\JPicScale mm}
\psset{linewidth=0.3,dotsep=1,hatchwidth=0.3,hatchsep=1.5,shadowsize=1,dimen=middle}
\psset{dotsize=0.7 2.5,dotscale=1 1,fillcolor=black}
\psset{arrowsize=1 2,arrowlength=1,arrowinset=0.25,tbarsize=0.7 5,bracketlength=0.15,rbracketlength=0.15}
\begin{pspicture}(0,0)(42.5,27.5)
\psline(6.25,21.25)(10,17.5)
\psline(3.75,21.25)(0,17.5)
\psline(5,3.75)(5,0)
\psline(0,17.5)(0,10)
\psline(10,17.5)(10,10)
\psline(0,10)(3.75,6.25)
\psline(10,10)(6.25,6.25)
\psline(5,27.5)(5,23.75)
\psline(42.5,27.5)(42.5,0)
\psline(3.75,23.75)(3.75,21.25)
\psline(3.75,21.25)(6.25,21.25)
\psline(3.75,23.75)(6.25,23.75)
\psline(6.25,21.25)(6.25,23.75)
\psline(3.75,6.25)(3.75,3.75)
\psline(3.75,3.75)(6.25,3.75)
\psline(3.75,6.25)(6.25,6.25)
\psline(6.25,3.75)(6.25,6.25)
\rput(5,5){$\copregr$}
\rput(5,22.5){$\pregr$}
\rput(15,13.75){\SUP}
\psline(28.75,21.25)(32.5,17.5)
\psline(26.25,21.25)(23.75,18.75)
\psline(27.5,3.75)(27.5,0)
\psline(32.5,17.5)(32.5,10)
\psline(23.75,8.75)(26.25,6.25)
\psline(32.5,10)(28.75,6.25)
\psline(27.5,27.5)(27.5,23.75)
\psline(26.25,23.75)(26.25,21.25)
\psline(26.25,21.25)(28.75,21.25)
\psline(26.25,23.75)(28.75,23.75)
\psline(28.75,21.25)(28.75,23.75)
\psline(26.25,6.25)(26.25,3.75)
\psline(26.25,3.75)(28.75,3.75)
\psline(26.25,6.25)(28.75,6.25)
\psline(28.75,3.75)(28.75,6.25)
\rput(27.5,5){$\copregr$}
\rput(27.5,22.5){$\pregr$}
\psline(23.75,18.75)(23.75,16.25)
\psline(23.75,11.25)(23.75,8.75)
\psline(21.25,11.25)(21.25,8.75)
\psline(21.25,18.75)(21.25,16.25)
\psline(21.25,11.25)(23.75,11.25)
\psline(21.25,8.75)(23.75,8.75)
\psline(21.25,16.25)(23.75,16.25)
\psline(21.25,18.75)(23.75,18.75)
\rput(22.5,17.5){$\prunt$}
\rput(22.5,10){$\prcun$}
\rput(37.5,13.75){\EQLS}
\end{pspicture}

\end{split}
\eeq
Putting together \eqref{eq:monun} with the second part of strictness yields the equality $\scomp \cmn \mnd = \id_A$, i.e.
\bea\label{eq:totinj}
\exists uv.\ \ \derrr x \cmn {u,v} \wedge \derrr {u,v} \mnd y & \iff & \derrr x A {y}
\eea
for all $x,y\in A$. In words, \eqref{eq:totinj} says that $\cmn: A\to A\times A$ is total and injective, and that $\mnd:A\times A\to A$ is single-valued and surjective.

\subsection{Adjunctions in monoids}\label{Sec:adj}
The crucial structure which makes a preordered monoid into a pregroup are the operations $\ell, r:A\to A$, called the $\ell$eft and the $r$ight adjunction, which satisfy the following inclusions
\beq\label{eq:adjoints}
\begin{split}
\newcommand{\pregr}{\scriptstyle \therefore}
\newcommand{\prunt}{\scriptstyle \unt}
\newcommand{\rugt}{\scriptstyle r}
\newcommand{\lugt}{\scriptstyle \ell}
\def\JPicScale{1.1}
\ifx\JPicScale\undefined\def\JPicScale{1}\fi
\psset{unit=\JPicScale mm}
\psset{linewidth=0.3,dotsep=1,hatchwidth=0.3,hatchsep=1.5,shadowsize=1,dimen=middle}
\psset{dotsize=0.7 2.5,dotscale=1 1,fillcolor=black}
\psset{arrowsize=1 2,arrowlength=1,arrowinset=0.25,tbarsize=0.7 5,bracketlength=0.15,rbracketlength=0.15}
\begin{pspicture}(0,0)(125,27.5)
\psline(86.25,21.25)(90,17.5)
\psline(83.75,21.25)(80,17.5)
\psline(85,5)(85,0)
\psline(80,17.5)(80,10)
\psline(80,10)(85,5)
\psline(90,10)(85,5)
\psline(85,27.5)(85,23.75)
\psline(102.5,27.5)(102.5,17.5)
\psline(83.75,23.75)(83.75,21.25)
\psline(83.75,21.25)(86.25,21.25)
\psline(83.75,23.75)(86.25,23.75)
\psline(86.25,21.25)(86.25,23.75)
\rput(85,22.5){$\pregr$}
\psline(121.25,21.25)(125,17.5)
\psline(118.75,21.25)(115,17.5)
\psline(120,5)(120,0)
\psline(125,17.5)(125,10)
\psline(115,10)(120,5)
\psline(125,10)(120,5)
\psline(120,27.5)(120,23.75)
\psline(118.75,23.75)(118.75,21.25)
\psline(118.75,21.25)(121.25,21.25)
\psline(118.75,23.75)(121.25,23.75)
\psline(121.25,21.25)(121.25,23.75)
\rput(120,22.5){$\pregr$}
\psline(115,17.5)(115,15)
\psline(116.25,15)(116.25,12.5)
\psline(113.75,15)(113.75,12.5)
\psline(103.75,17.5)(103.75,15)
\psline(113.75,15)(116.25,15)
\psline(113.75,12.5)(116.25,12.5)
\psline(101.25,15)(103.75,15)
\psline(101.25,17.5)(103.75,17.5)
\rput(102.5,16.25){$\prunt$}
\psline(102.5,11.25)(102.5,0)
\psline(101.25,17.5)(101.25,15)
\psline(115,12.5)(115,10)
\psline(90,17.5)(90,15)
\psline(91.25,15)(91.25,12.5)
\psline(88.75,15)(88.75,12.5)
\psline(88.75,15)(91.25,15)
\psline(88.75,12.5)(91.25,12.5)
\psline(90,12.5)(90,10)
\rput(90,13.75){$\rugt$}
\rput(115,13.75){$\rugt$}
\rput(108.75,13.75){\SUB}
\rput(95,13.75){\SUB}
\rput{0}(102.5,11.25){\psellipse[fillstyle=solid](0,0)(1.09,-1.09)}
\rput{0}(120,5){\psellipse[fillstyle=solid](0,0)(1.09,-1.09)}
\rput{0}(85,5){\psellipse[fillstyle=solid](0,0)(1.09,-1.09)}
\psline(41.25,21.25)(45,17.5)
\psline(38.75,21.25)(35,17.5)
\psline(40,5)(40,0)
\psline(35,17.5)(35,10)
\psline(35,10)(40,5)
\psline(45,10)(40,5)
\psline(40,27.5)(40,23.75)
\psline(22.5,27.5)(22.5,17.5)
\psline(38.75,23.75)(38.75,21.25)
\psline(38.75,21.25)(41.25,21.25)
\psline(38.75,23.75)(41.25,23.75)
\psline(41.25,21.25)(41.25,23.75)
\rput(40,22.5){$\pregr$}
\psline(6.25,21.25)(10,17.5)
\psline(3.75,21.25)(0,17.5)
\psline(5,5)(5,0)
\psline(10,17.5)(10,10)
\psline(0,10)(5,5)
\psline(10,10)(5,5)
\psline(5,27.5)(5,23.75)
\psline(3.75,23.75)(3.75,21.25)
\psline(3.75,21.25)(6.25,21.25)
\psline(3.75,23.75)(6.25,23.75)
\psline(6.25,21.25)(6.25,23.75)
\rput(5,22.5){$\pregr$}
\psline(0,17.5)(0,15)
\psline(1.25,15)(1.25,12.5)
\psline(-1.25,15)(-1.25,12.5)
\psline(23.75,17.5)(23.75,15)
\psline(-1.25,15)(1.25,15)
\psline(-1.25,12.5)(1.25,12.5)
\psline(21.25,15)(23.75,15)
\psline(21.25,17.5)(23.75,17.5)
\rput(22.5,16.25){$\prunt$}
\psline(22.5,11.25)(22.5,0)
\psline(21.25,17.5)(21.25,15)
\psline(0,12.5)(0,10)
\psline(45,17.5)(45,15)
\psline(46.25,15)(46.25,12.5)
\psline(43.75,15)(43.75,12.5)
\psline(43.75,15)(46.25,15)
\psline(43.75,12.5)(46.25,12.5)
\psline(45,12.5)(45,10)
\rput(28.75,13.75){\SUB}
\rput(15,13.75){\SUB}
\rput{0}(22.5,11.25){\psellipse[fillstyle=solid](0,0)(1.09,-1.09)}
\rput{0}(5,5){\psellipse[fillstyle=solid](0,0)(1.09,-1.09)}
\rput{0}(40,5){\psellipse[fillstyle=solid](0,0)(1.09,-1.09)}
\rput(45,13.75){$\lugt$}
\rput(0,13.75){$\lugt$}
\end{pspicture}

\end{split}
\eeq
Like the monoid operation $\mnd$, the adjunction operations are assumed to be representable, i.e. for every $x\in A$ there are $x^\ell, x^r\in A$ such that
\beq \derrr x \ell y \iff \derr {x^\ell}  y \qquad \qquad \qquad\qquad  \derrr x r y \iff \derr {x^r} y\eeq
Since everything in sight  in \eqref{eq:adjoints} is thus representable, the inclusions boil down to the implications
\beq\derr {x^\ell \cdot x} u  \implies  \derr \unt u  \implies \derr {x\cdot x^\ell} u  \qquad \quad
\derr {x\cdot x^r} v  \implies  \derr \unt v  \implies \derr {x^r\cdot x} v\eeq
Substituting $x^\ell \cdot x$ for $u$, the first implication gives $\derr \unt {x^\ell \cdot x}$. Substituting $\unt$, the second implication gives $\derr {x\cdot x^\ell} \unt$. Proceeding with the third and the fourth implication in a similar way, and putting it all together gives the full set of adjunction sequents
\beq\label{eq:adj} {x x^\ell} \leftarrow\, \unt \leftarrow {x^\ell  x}  \qquad \qquad \qquad \qquad
{x^r  x} \leftarrow\, \unt \leftarrow {x x^r} \eeq
where we elide the dot and write $ab$ for $a\cdot b$.

\paragraph{\sf Remark.} 
Pregroups have been defined and used as partially ordered rather than preordered monoids, i.e. satisfying $\derreq x y \implies x=y$. The justification for reverting to preorders is given at the beginning of Sec.~\ref{Sec:repres}.

\paragraph{\sf Examples.} For applications in linguistics, the pregroup grammars are freely generated from some basic syntactic types. The elements of a free pregroup are sequences of literals, corresponding to all derived syntactic types, possibly annotated by sequences of a superscript $\ell$ or $r$. The ordering is generated by \eqref{eq:adj}. The observations in the next section will carve out a rich source of pregroups ``in nature''. An interesting pair of an example and a non-example, going back to \cite{Lambek-Moser}, which informed the introduction of pregroups into  algebra, was presented in \cite{LambekJ:Galois}. 

\subsection{Pregroups are properties}\label{Sec:property}
It is often convenient to expand the adjunction sequents in \eqref{eq:adj} to adjunction correspondences
\beq \label{eq:adjj}
\derr {x a} b \iff \derr a {x^\ell b} \qquad \qquad \qquad\qquad \derr {x^r a} b \iff \derr a {x b}
\eeq
valid for all $a,b\in A$. Conditions \eqref{eq:adj} and \eqref{eq:adjj} are equivalent characterizations of the adjunction operations $\ell$ and $r$. The third equivalent characterization is
\beq \label{eq:adjjj}
x^\ell\leftrightarrow \bigvee\big\{a|\derr {xa} \unt \big\} \qquad \qquad \qquad \qquad x^r \leftrightarrow \bigwedge\big\{b|\derr {\unt} {xb}\big\}
\eeq
These characterizations of adjunctions are well-known. We mention them to emphasize that being a  pregroup is an intrinsic \emph{property}\/ that a monoid may have, and not an extrinsic \emph{structure}\/ that may be added to it. E.g., if every element of a monoid has an inverse, then it is a group, and this is its intrinsic property. On the other hand, an addtive group may support many different multiplication operations, and can therefore be extended into a ring in many different ways, extrinsically. Ditto for an additive monoid, which can be extended into a rig\footnote{A \emph{rig}\/ is a ri\textbf{n}g without the \textbf{n}egatives. The same structure is also denoted by the less descriptive name \emph{semiring}.} in many different ways.  For a monoid, being a group is thus a property, whereas being a rig is a structure. A preordered monoid $(A,\mnd, \unt)$ may be a pregroup in at most one way, in if and only if for every $x$ it has
\begin{itemize}
\item suprema of sets in the form $\big\{a|\derr {xa} \unt \big\}$, and
\item infima of sets in the form $\big\{b|\derr {\unt} {xb}\big\}$.
\end{itemize}
Being a pregroup is thus a \emph{completeness property}\/ of a preordered monoid. E.g., all monoid structures over complete lattices are pregroups. Extending a preordered monoid into a pregroup is a \emph{completion}\/ of its preorder. The other way around, the pregroup operations also completely determine the pregroup preordering, since \eqref{eq:adjj} implies
\beq
\derr \unt {x^{\ell}y} \ \ \iff \ \ \derr x y \ \ \iff \ \ \derr {y^{r}x} \unt
\eeq
A 'positive cone' of a pregroup $A$, i.e. a set of elements $H = \{h| \derr h \unt\}$, can be characterized without mentioning the preordering as a submonoid of the monoid $A$ such that $h\in H$ implies $x^{r}hx\in H$ for all $x\in A$ and moreover $x,x^{r}\in H$ only if $\derreq x \unt$. Such a submonoid completely determines the preorder by $\derr x y \iff y^{r}x\in H$.

\smallskip 
The crucial property that drives the applications of pregroups as algebra of sequents suitable for presenting the syntax of natural languages is that their operations, the adjoints, close the open sequents in the form $\derr {x\cdots} \unt$ and $\derr {\unt} {x\cdots}$, which can be construed as capturing the processes of generation and parsing. This turns out to be an inherent capability of the spider algebras as well.

\section{Spiders}\label{Sec:frobspid}

The name \emph{spider}\/ has been used\footnote{The spiders emerged among the coauthors of \cite{PavlovicD:CQStruct} as a nickname for the "special Frobenius algebras" of \cite{PavlovicD:QMWS}. They diversified through the many contexts of the monograph \cite{CoeckeB:Picturing}.} to denote the monoid-comonoid couples  in monoidal categories that give rise to string diagrams in the general form
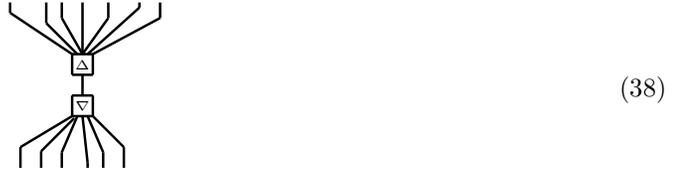
\beq\label{eq:spider}
\begin{split}
\newcommand{\fmnd}{\scriptscriptstyle \spmnd}
\newcommand{\fcmn}{\scriptscriptstyle \spcmn}
\def\JPicScale{1.1}
\ifx\JPicScale\undefined\def\JPicScale{1}\fi
\psset{unit=\JPicScale mm}
\psset{linewidth=0.3,dotsep=1,hatchwidth=0.3,hatchsep=1.5,shadowsize=1,dimen=middle}
\psset{dotsize=0.7 2.5,dotscale=1 1,fillcolor=black}
\psset{arrowsize=1 2,arrowlength=1,arrowinset=0.25,tbarsize=0.7 5,bracketlength=0.15,rbracketlength=0.15}
\begin{pspicture}(0,0)(19.38,20)
\psline(8.75,8.75)(8.75,6.25)
\psline(11.25,6.25)(15,2.5)
\psline(8.75,6.25)(2.5,2.5)
\psline(10,11.25)(10,8.75)
\psline(8.75,6.25)(11.25,6.25)
\psline(8.75,8.75)(11.25,8.75)
\psline(11.25,6.25)(11.25,8.75)
\psline(2.5,2.5)(2.5,0)
\psline(15,2.5)(15,0)
\psline(1.25,18.75)(8.75,13.75)
\psline(16.88,19.38)(10.62,13.75)
\psline(5.62,20)(5.62,17.5)
\psline(16.88,20)(16.88,19.38)
\psline(8.75,13.75)(8.75,11.25)
\psline(8.75,11.25)(11.25,11.25)
\psline(8.75,13.75)(11.25,13.75)
\psline(11.25,11.25)(11.25,13.75)
\psline(10.62,6.25)(12.5,1.88)
\psline(9.38,6.25)(7.5,1.88)
\psline(9,6)(5,2)
\psline(7.5,1.88)(7.5,0)
\psline(12.5,1.88)(12.5,0)
\psline(10,13.75)(13.12,18.12)
\psline(7.5,18.12)(7.5,20)
\psline(13.12,20)(13.12,18.12)
\psline(10,13.75)(7.5,18.12)
\psline(10,20)(10,13.75)
\psline(5,2)(5,0)
\psline(10,6.25)(10.62,0.62)
\psline(10.62,0.62)(10.62,0)
\psline(5.62,17.5)(9.38,13.75)
\psline(1.25,20)(1.25,18.75)
\psline(19.38,18.12)(11.25,13.75)
\psline(19.38,20)(19.38,18.12)
\rput(10,12.5){$\fcmn$}
\rput(10,7.5){$\fmnd$}
\end{pspicture}

\end{split}
\eeq
A spider may have any number of legs. The front legs are spanned by the comonoid operation, the hind legs by the monoid operation. In each case, the multiple applications of the operation, are "flattened out" using the associativity. The crucial feature of spider algebra is the normalization to the shape of a spider, whereby all applications of the monoid operation are  moved before all applications of the comonoid operation. The string diagrams representing the spider normalization laws are displayed in (\ref{eq:Frobspec-1}--\ref{eq:Frobspec-2}). The sequent rewrites corresponding to the spider laws in $\PRel$ are spelled out above the diagrams. Formally, a spider is thus a monoid-comonoid couple satisfying the following conditions. 

\beq\label{eq:Frobspec-1}
\begin{split}
\newcommand{\eqone}{\scriptstyle \exists s.\derrr {x,s} \spmnd u \wedge \derrr y \spcmn {s,v}} 
\newcommand{\eqtwo}{\scriptstyle \exists w. \derrr{x,y} \spmnd w \wedge \derrr w \spcmn {u,v}}
\newcommand{\eqthree}{\scriptstyle \exists t. \derrr x \spcmn  {u, t} \wedge \derrr{t,y} \spmnd  v }
\newcommand{\eqfour}{\scriptstyle \exists uv. \derrr x \spcmn {u, v}\wedge \derrr {u,v} \spmnd y}
\newcommand{\eqfive}{\scriptstyle \derrr x A y}
\newcommand{\eex}{\scriptstyle x}
\newcommand{\eey}{\scriptstyle y}
\newcommand{\eet}{\scriptstyle t}
\newcommand{\eeu}{\scriptstyle u}
\newcommand{\eev}{\scriptstyle v}
\newcommand{\eew}{\scriptstyle w}
\newcommand{\ees}{\scriptstyle s}
\newcommand{\eez}{\scriptstyle z}
\newcommand{\fmnd}{\scriptscriptstyle \spmnd}
\newcommand{\fcmn}{\scriptscriptstyle \spcmn}
\def\JPicScale{1.1}
\ifx\JPicScale\undefined\def\JPicScale{1}\fi
\psset{unit=\JPicScale mm}
\psset{linewidth=0.3,dotsep=1,hatchwidth=0.3,hatchsep=1.5,shadowsize=1,dimen=middle}
\psset{dotsize=0.7 2.5,dotscale=1 1,fillcolor=black}
\psset{arrowsize=1 2,arrowlength=1,arrowinset=0.25,tbarsize=0.7 5,bracketlength=0.15,rbracketlength=0.15}
\begin{pspicture}(0,0)(100,27.5)
\rput(10,27.5){$\eqone$}
\rput(50,27.5){$\eqtwo$}
\psline(15,20)(15,12.5)
\psline(10,6.25)(10,0)
\psline(0.01,7.5)(0,0)
\psline(15,12.5)(11.25,8.75)
\psline(0.01,7.5)(3.76,11.25)
\psline(6.26,11.25)(8.76,8.75)
\psline(5,20)(4.99,13.75)
\rput(90,27.5){$\eqthree$}
\rput[l](6.25,20){$\eeu$}
\rput[l](16.25,20){$\eev$}
\rput[r](-1.25,0){$\eex$}
\rput[r](8.75,0){$\eey$}
\psline(3.76,13.75)(6.26,13.75)
\psline(3.76,11.25)(6.26,11.25)
\psline(6.26,13.75)(6.26,11.25)
\psline(3.76,13.75)(3.76,11.25)
\psline(8.76,8.75)(11.26,8.75)
\psline(8.76,6.25)(11.26,6.25)
\psline(11.26,8.75)(11.26,6.25)
\psline(8.76,8.75)(8.76,6.25)
\psline(48.75,8.75)(48.75,6.25)
\psline(51.25,6.25)(55,2.5)
\psline(48.75,6.25)(45,2.5)
\psline(50,11.25)(50,8.75)
\psline(48.75,6.25)(51.25,6.25)
\psline(48.75,8.75)(51.25,8.75)
\psline(51.25,6.25)(51.25,8.75)
\psline(45,17.5)(48.75,13.75)
\psline(55,17.5)(51.25,13.75)
\psline(45,20)(45,17.5)
\psline(55,20)(55,17.5)
\rput(70,10){\SUB}
\psline(48.75,13.75)(48.75,11.25)
\psline(48.75,11.25)(51.25,11.25)
\psline(48.75,13.75)(51.25,13.75)
\psline(51.25,11.25)(51.25,13.75)
\psline(45,2.5)(45,0)
\psline(55,2.5)(55,0)
\rput(30,10){\SUP}
\psline(85,20)(85,12.5)
\psline(90,6.25)(90,0)
\psline(100,7.5)(100,0)
\psline(85,12.5)(88.75,8.75)
\psline(91.25,8.75)(93.75,11.25)
\psline(96.25,11.25)(100,7.5)
\psline(95,20)(94.99,13.75)
\psline(93.75,13.75)(96.25,13.75)
\psline(93.75,11.25)(96.25,11.25)
\psline(96.25,13.75)(96.25,11.25)
\psline(93.75,13.75)(93.75,11.25)
\psline(88.75,8.75)(91.25,8.75)
\psline(88.75,6.25)(91.25,6.25)
\psline(91.25,8.75)(91.25,6.25)
\psline(88.75,8.75)(88.75,6.25)
\rput[bl](8.12,10.62){$\ees$}
\rput[br](91.88,10.62){$\eet$}
\rput[l](46.25,20){$\eeu$}
\rput[l](56.25,20){$\eev$}
\rput[l](86.25,20){$\eeu$}
\rput[l](96.25,20){$\eev$}
\rput[r](43.75,0){$\eex$}
\rput[r](53.75,0){$\eey$}
\rput[r](88.75,0){$\eex$}
\rput[r](98.75,0){$\eey$}
\rput(70,27.5){\THEN}
\rput(30,27.5){\IF}
\rput[l](51.25,10){$\eew$}
\rput(5,12.5){$\fmnd$}
\rput(90,7.5){$\fcmn$}
\rput(50,12.5){$\fcmn$}
\rput(10,7.5){$\fcmn$}
\rput(50,7.5){$\fmnd$}
\rput(95,12.5){$\fmnd$}
\end{pspicture}

\end{split}
\eeq

\bigskip
\beq\label{eq:Frobspec-2}
\begin{split}
\newcommand{\eqone}{\scriptstyle \exists s.\derrr {x,s} \spmnd u \wedge \derrr y \spcmn {s,v}} 
\newcommand{\eqtwo}{\scriptstyle \exists z. \derrr{x,y} \spmnd z \wedge \derrr z \spcmn {u,v}}
\newcommand{\eqthree}{\scriptstyle \exists t. \derrr x \spcmn  {u,t} \wedge \derrr{t,y} \spmnd  v }
\newcommand{\eqfour}{\scriptstyle \exists uv. \derrr x \spcmn {u, v}\wedge \derrr {u,v} \spmnd y}
\newcommand{\eqfive}{\scriptstyle \derrr x A y}
\newcommand{\eex}{\scriptstyle x}
\newcommand{\eey}{\scriptstyle y}
\newcommand{\eet}{\scriptstyle t}
\newcommand{\eeu}{\scriptstyle u}
\newcommand{\eev}{\scriptstyle v}
\newcommand{\eew}{\scriptstyle w}
\newcommand{\ees}{\scriptstyle s}
\newcommand{\eez}{\scriptstyle z}
\newcommand{\fmnd}{\scriptscriptstyle \spmnd}
\newcommand{\fcmn}{\scriptscriptstyle \spcmn}
\def\JPicScale{1.1}
\ifx\JPicScale\undefined\def\JPicScale{1}\fi
\psset{unit=\JPicScale mm}
\psset{linewidth=0.3,dotsep=1,hatchwidth=0.3,hatchsep=1.5,shadowsize=1,dimen=middle}
\psset{dotsize=0.7 2.5,dotscale=1 1,fillcolor=black}
\psset{arrowsize=1 2,arrowlength=1,arrowinset=0.25,tbarsize=0.7 5,bracketlength=0.15,rbracketlength=0.15}
\begin{pspicture}(0,0)(33.75,27.5)
\rput(2.5,27.5){$\eqfour$}
\rput(32.5,27.5){$\eqfive$}
\psline(5,15)(7.5,12.5)
\psline(2.5,15)(0,12.5)
\psline(3.75,2.5)(3.75,0)
\psline(0,12.5)(0,7.5)
\psline(7.5,12.5)(7.5,7.5)
\psline(0,7.5)(2.5,5)
\psline(7.5,7.5)(5,5)
\psline(32.5,20)(32.5,0)
\rput(22.5,10){\SUB}
\psline(2.5,17.5)(2.5,15)
\psline(2.5,15)(5,15)
\psline(2.5,17.5)(5,17.5)
\psline(5,15)(5,17.5)
\psline(2.5,5)(2.5,2.5)
\psline(2.5,2.5)(5,2.5)
\psline(2.5,5)(5,5)
\psline(5,2.5)(5,5)
\rput(22.5,27.5){\THEN}
\rput[r](-1.25,10){$\eeu$}
\rput[l](8.75,10){$\eev$}
\rput[r](2.5,0){$\eex$}
\rput[l](5,20){$\eey$}
\psline(3.75,20)(3.75,17.5)
\rput[r](31.25,0){$\eex$}
\rput[l](33.75,20){$\eey$}
\rput(3.75,3.75){$\fcmn$}
\rput(3.75,16.25){$\fmnd$}
\end{pspicture}

\end{split}
\eeq
It is easy to see that the implications and containments dual to those displayed in (\ref{eq:Frobspec-1}--\ref{eq:Frobspec-2}) are satisfied in all monoids: those in \eqref{eq:Frobspec-1} by the associativity, the one on \eqref{eq:Frobspec-2} because of the monoid unit. This was discussed at the end of Sec.~\ref{Sec:mon}. Strengthened to equations, the conditions in \eqref{eq:Frobspec-1} are known as the  \emph{Frobenius conditions}, for reasons explained in \cite{LawvereFW:equhcs,PavlovicD:mapsII}, and the structures satisfying them are called \emph{Frobenius algebras}   \cite{CarboniA:matrices,Carboni-Walters,PavlovicD:QMWS,PavlovicD:Qabs12,StreetR:Frobenius}. The equational version of the condition in \eqref{eq:Frobspec-2} makes the induced convolution operation isometric, and the Frobenius algebras satisfying it are called \emph{special}, or shortly \emph{spiders}. The sequent form in (\ref{eq:Frobspec-1}--\ref{eq:Frobspec-2}) displays the logical content of the spider equations without the implications that are satisfied by all monoids. Since the eliminated implications are satisfied automatically, the presented analyses apply to the standard equational definitions and intuitions without a gain in generality.

\section{Pregroups are pointed spiders}\label{Sec:thm}


\begin{theorem}\label{thm:one}
A monoid in $\PRel$ is a pregroup if and only it is pointed in the sense of \eqref{eq:untmap} and a spider in the sense of {\rm (\ref{eq:Frobspec-1}--\ref{eq:Frobspec-2})}.
\end{theorem}

\begin{proof}
Suppose that $\spmnd = \mnd:A\times A \to A$ is a pregroup operation, with the dual $\spcmn = \cmn:A\prel A\times A$. Since $\spmnd$ is thus a map, (\ref{eq:represpreg}--\ref{eq:pointed}) yield the representations
\begin{align*}
\derrr {x,y}\spmnd u & \iff \  \derrr{\mult x y} A u&&& \derrr u\spcmn {x,y} & \iff \  \derrr u A {\mult x y}\\
\derrr \emptyset \unt u & \iff \  \derrr \unt A u&&& \derrr u\cun \emptyset & \iff \  \derrr u A \unt\notag
\end{align*}
The assumption $\derrr {x,y}\spmnd z \wedge \derrr z \spcmn {u,v}$ from \eqref{eq:Frobspec-1} thus becomes $\derrr {xy} A z \wedge \derrr z A {uv}$, where we elide the dots. By transitivity, and eliding the type $A$, this boils down to $\derr{xy}{uv}$. To prove that the conclusions of the Frobenius condition are satisfied, we need to construct 
\begin{enumerate}[i)]
\item $s$ such that $\derr {xs} u$ and $\derr y {sv}$, and
\item $t$ such that $\derr x {ut}$ and $\derr {ty} v$.
\end{enumerate}
Since $A$ is a pregroup, we use the adjoints to define $s = x^\ell u$ and $t= vy^r$. Together with $\derr{xy}{uv}$, the sequents $\derr \unt {x^\ell x}$ and $\derr \unt {yy^r}$ yield 
\begin{align}
\left(y \leftarrow x^\ell x y \leftarrow x^\ell u v = sv\right) &&\mbox{ and }& \left( xs= xx^\ell u\leftarrow u\right)\tag {i}\\
\left(x \leftarrow xyy^r \leftarrow u vy^r = ut\right) &&\mbox{ and }& \left( ty= vy^r y\leftarrow v\right)\tag{ii}
\end{align}
To complete the proof that the pregroup $A$ is a pointed spider, it remains to be proved that the monoid satisfies the "specialty" condition \eqref{eq:Frobspec-2} and that the point requirement \eqref{eq:untmap}. Both conditions are satisfied by any representable monoid, since they correspond to
\bea \exists uv. \left(x\leftarrow uv \leftarrow y\right) &\implies & \derr x y \\
\derr x \unt \wedge \derr \unt y  &\implies& \derr x y\label{eq:totproof}
 \eea
This completes the proof that a pregroup structure induces a pointed spider.

\smallskip
The other way around, suppose that a dual monoid-comonoid couple $(A,\spmnd, \spcmn, \spunt, \spcun)$ in $\PRel$ satisfy conditions (\ref{eq:Frobspec-1}--\ref{eq:Frobspec-2}) for a spider, and also the point condition  \eqref{eq:untmap}. The Frobenius conditions \eqref{eq:Frobspec-1} and the point condition \eqref{eq:untmap} make $\spmnd$ total, as follows
\beq\label{eq:strict}
\begin{split}
\newcommand{\fmnd}{\scriptscriptstyle \spmnd}
\newcommand{\fcmn}{\scriptscriptstyle \spcmn}
\newcommand{\funt}{\scriptscriptstyle \spunt}
\newcommand{\fcun}{\scriptscriptstyle \spcun}
\def\JPicScale{1.1}
\ifx\JPicScale\undefined\def\JPicScale{1}\fi
\psset{unit=\JPicScale mm}
\psset{linewidth=0.3,dotsep=1,hatchwidth=0.3,hatchsep=1.5,shadowsize=1,dimen=middle}
\psset{dotsize=0.7 2.5,dotscale=1 1,fillcolor=black}
\psset{arrowsize=1 2,arrowlength=1,arrowinset=0.25,tbarsize=0.7 5,bracketlength=0.15,rbracketlength=0.15}
\begin{pspicture}(0,0)(132.5,30)
\psline(127.5,13.75)(127.5,11.25)
\psline(130,11.25)(132.5,8.75)
\psline(127.5,11.25)(125,8.75)
\psline(128.75,16.25)(128.75,13.75)
\psline(127.5,11.25)(130,11.25)
\psline(127.5,13.75)(130,13.75)
\psline(130,11.25)(130,13.75)
\psline(125,21.25)(127.5,18.75)
\psline(132.5,21.25)(130,18.75)
\psline(125,30)(125,21.25)
\psline(132.5,30)(132.5,21.25)
\rput(56.25,15){\SUB}
\psline(127.5,18.75)(127.5,16.25)
\psline(127.5,16.25)(130,16.25)
\psline(127.5,18.75)(130,18.75)
\psline(130,16.25)(130,18.75)
\psline(125,8.75)(125,0)
\psline(132.5,8.75)(132.5,0)
\psline(17.5,30)(17.5,16.25)
\psline(21.25,11.25)(21.25,0)
\psline(25,18.75)(25,16.25)
\psline(17.5,16.25)(20,13.75)
\psline(22.5,13.75)(25,16.25)
\psline(30,21.25)(32.5,18.75)
\psline(28.75,25)(28.74,23.75)
\psline(27.5,23.75)(30,23.75)
\psline(27.5,21.25)(30,21.25)
\psline(30,23.75)(30,21.25)
\psline(27.5,23.75)(27.5,21.25)
\psline(20,13.75)(22.5,13.75)
\psline(20,11.25)(22.5,11.25)
\psline(22.5,13.75)(22.5,11.25)
\psline(20,13.75)(20,11.25)
\rput(21.25,12.5){$\fcmn$}
\rput(128.75,17.5){$\fcmn$}
\rput(128.75,12.5){$\fmnd$}
\rput(28.75,22.5){$\fmnd$}
\psline(27.5,27.5)(30,27.5)
\psline(27.5,25)(30,25)
\psline(30,27.5)(30,25)
\psline(27.5,27.5)(27.5,25)
\psline(32.5,18.75)(35,18.75)
\psline(32.5,16.25)(35,16.25)
\psline(35,18.75)(35,16.25)
\psline(32.5,18.75)(32.5,16.25)
\psline(32.5,13.75)(35,13.75)
\psline(32.5,11.25)(35,11.25)
\psline(35,13.75)(35,11.25)
\psline(32.5,13.75)(32.5,11.25)
\psline(46.25,30)(46.25,18.75)
\psline(38.75,6.25)(38.75,5)
\psline(35,11.25)(37.5,8.75)
\psline(40,8.75)(42.5,11.25)
\psline(47.5,16.25)(50,13.75)
\psline(45,18.75)(47.5,18.75)
\psline(45,16.25)(47.5,16.25)
\psline(47.5,18.75)(47.5,16.25)
\psline(45,18.75)(45,16.25)
\psline(37.5,8.75)(40,8.75)
\psline(37.5,6.25)(40,6.25)
\psline(40,8.75)(40,6.25)
\psline(37.5,8.75)(37.5,6.25)
\rput(38.75,7.5){$\fcmn$}
\rput(46.25,17.5){$\fmnd$}
\psline(37.5,5)(40,5)
\psline(37.5,2.5)(40,2.5)
\psline(40,5)(40,2.5)
\psline(37.5,5)(37.5,2.5)
\psline(25,18.75)(27.5,21.25)
\psline(42.5,13.75)(45,16.25)
\psline(42.5,13.75)(42.5,11.25)
\psline(50,13.75)(50,0)
\psline(62.5,30)(62.5,16.25)
\psline(66.25,11.25)(66.25,0)
\psline(70,18.75)(70,16.25)
\psline(62.5,16.25)(65,13.75)
\psline(67.5,13.75)(70,16.25)
\psline(75,21.25)(77.5,18.75)
\psline(73.75,25)(73.74,23.75)
\psline(72.5,23.75)(75,23.75)
\psline(72.5,21.25)(75,21.25)
\psline(75,23.75)(75,21.25)
\psline(72.5,23.75)(72.5,21.25)
\psline(65,13.75)(67.5,13.75)
\psline(65,11.25)(67.5,11.25)
\psline(67.5,13.75)(67.5,11.25)
\psline(65,13.75)(65,11.25)
\rput(66.25,12.5){$\fcmn$}
\rput(73.75,22.5){$\fmnd$}
\psline(72.5,27.5)(75,27.5)
\psline(72.5,25)(75,25)
\psline(75,27.5)(75,25)
\psline(72.5,27.5)(72.5,25)
\psline(88.75,30)(88.75,18.75)
\psline(81.25,6.25)(81.25,5)
\psline(77.5,11.25)(80,8.75)
\psline(82.5,8.75)(85,11.25)
\psline(90,16.25)(92.5,13.75)
\psline(87.5,18.75)(90,18.75)
\psline(87.5,16.25)(90,16.25)
\psline(90,18.75)(90,16.25)
\psline(87.5,18.75)(87.5,16.25)
\psline(80,8.75)(82.5,8.75)
\psline(80,6.25)(82.5,6.25)
\psline(82.5,8.75)(82.5,6.25)
\psline(80,8.75)(80,6.25)
\rput(81.25,7.5){$\fcmn$}
\rput(88.75,17.5){$\fmnd$}
\psline(80,5)(82.5,5)
\psline(80,2.5)(82.5,2.5)
\psline(82.5,5)(82.5,2.5)
\psline(80,5)(80,2.5)
\psline(70,18.75)(72.5,21.25)
\psline(85,13.75)(87.5,16.25)
\psline(85,13.75)(85,11.25)
\psline(92.5,13.75)(92.5,0)
\psline(77.5,18.75)(77.5,11.25)
\psline(102.5,30)(102.5,16.25)
\psline(106.25,11.25)(106.25,0)
\psline(102.5,16.25)(105,13.75)
\psline(105,13.75)(107.5,13.75)
\psline(105,11.25)(107.5,11.25)
\psline(107.5,13.75)(107.5,11.25)
\psline(105,13.75)(105,11.25)
\rput(106.25,12.5){$\fcmn$}
\psline(111.25,30)(111.25,18.75)
\psline(112.5,16.25)(115,13.75)
\psline(110,18.75)(112.5,18.75)
\psline(110,16.25)(112.5,16.25)
\psline(112.5,18.75)(112.5,16.25)
\psline(110,18.75)(110,16.25)
\rput(111.25,17.5){$\fmnd$}
\psline(107.5,13.75)(110,16.25)
\psline(115,13.75)(115,0)
\psline(6.25,30)(6.25,0)
\psline(0,30)(0,0)
\rput(11.25,15){\EQLS}
\rput(97.5,15){\EQLS}
\rput(120,15){\EQLS}
\rput(81.25,3.75){$\funt$}
\rput(38.75,3.75){$\funt$}
\rput(33.75,17.5){$\funt$}
\rput(33.75,12.5){$\fcun$}
\rput(73.75,26.25){$\fcun$}
\rput(28.75,26.25){$\fcun$}
\end{pspicture}

\end{split}
\eeq
Hence the left-hand part of the strictness \eqref{eq:monmap}. The right-hand part of the strictness condition is just the specialty condition \eqref{eq:Frobspec-2}. Since the monoid is thus strict, and by the assumption in \eqref{eq:untmap} also pointed, we have the representatives in (\ref{eq:represpreg}--\ref{eq:pointed}) again. To prove that the monoid $\spmnd: A\times A \to A$ is a pregroup, we derive the adjoints satisfying \eqref{eq:adj} from the Frobenius conditions in \eqref{eq:Frobspec-1}. For a representable monoid, these conditions take the form
\bea
\derr {xy} {uv} & \implies & \exists s.\Big(\derr {xs} u \wedge \derr y {sv} \Big) \label{eq:err} \\
\derr {xy} {uv} &\implies&  \exists t. \Big( \derr x {ut} \wedge \derr{ty} v  \Big)  \label{eq:ell}
\eea
For $v=x$ and $u=y=\unt=\spunt$, \eqref{eq:err} becomes $\derr {x} {x} \implies \exists s.\Big(\derr {xs} \unt \wedge \derr \unt {sx} \Big)$. Hence $x^\ell = s$. Setting $x=v=\unt=\spunt$ and $y = u$, \eqref{eq:ell} boils down to $\derr {y} {y} \implies \exists t. \Big( \derr \unt {yt} \wedge \derr{ty} \unt  \Big)$. Hence $y^r = t$. Since $x$ and $y$ are arbitrary, the statement is proved.
\end{proof}
%

\section{Frobenius algebras are residuated  monoids}\label{Sec:residuated}

The tight alignment of the structures of pregroups and spiders depends  essentially on the assumptions which make them into maps. While convenient, these assumptions are not always natural for sequent derivations, and preclude some applications. Towards a finer alignment of the structures, and a more flexible derivational framework, in this section we characterize the Frobenius conditions \eqref{eq:Frobspec-1} in terms of the monoid residue operations, on which Lambek's categorial grammars were based \cite{LambekJ:math-sentence}. In the next section we shall add the isometry conditions \eqref{eq:Frobspec-2} back in, and return to the spiders.

A preordered monoid is said to be \emph{residuated}\/ when its operation admits the adjoints, which are for monoids called the residuations, because that is what they are for the extensions of modular multiplication monoids. Residuated monoids were first studied in \cite{Ward-Dilworth}, and the monoids of additive subgroups of rings ordered by inclusion appear in  \cite[\S 1.2]{LambekJ:rings}. Given a prelational monoid-comonoid dual couple $(A,\mnd,\cmn,\unt,\cun)$, we define 
\begin{itemize}
\item the \emph{right residuations}\/ as the mappings $\Rr,\Lr:\Up A\to \Do A$ such that
\beq\label{eq:radj-res}
\derrr {\xi , y} \mnd z \iff \derrr y \cmn {\xi^{\Rr}, z}\qquad \qquad\qquad \qquad \derrr {x, \upsilon} \mnd z \iff \derrr x \cmn {z, \upsilon^{\Lr}}  
\eeq
hold for all $x,y,z\in A$ and $\xi,\upsilon \in \Up A$;
\item the \emph{left residuations}\/ as the mappings $\Ll,\Rl:\Do A\to \Up A$ such that
\beq\label{eq:ladj-res}
\derrr {x} \cmn {y,\zeta } \iff \derrr {x,\zeta^\Ll} \mnd {y} \qquad \qquad\qquad \qquad 
\derrr {x} \cmn {\upsilon, z} \iff \derrr {\upsilon^\Rl, x} \mnd {z}
\eeq
hold for all $x,y,z\in A$ and $\upsilon,\zeta \in \Up A$
\end{itemize}


\begin{theorem}\label{thm:two}
A monoid in $\PRel$ satisfies the Frobenius condition if and only if it is left residuated, or equivalently if and only if it is right residuated.
\end{theorem}

\begin{proof}
Given a Frobenius algebra $(A,\spmnd,\spcmn,\spunt,\spcun)$, the right residue operations can be defined as
\medskip
\beq\label{eq:residues}
\begin{split}
\newcommand{\xil}{\xi^{\Rr} = \Big((\xi\times A);\spmnd;\spcun\Big)}
\newcommand{\xir}{\upsilon^{\Lr} = \Big(A\times\upsilon);\spmnd;\spcun\Big)}
\newcommand{\pregr}{\scriptscriptstyle \spmnd}
\newcommand{\prunt}{\scriptscriptstyle \spcun}
\newcommand{\lugt}{\scriptstyle \xi}
\newcommand{\rugt}{\scriptstyle \upsilon}
\def\JPicScale{1.1}
\ifx\JPicScale\undefined\def\JPicScale{1}\fi
\psset{unit=\JPicScale mm}
\psset{linewidth=0.3,dotsep=1,hatchwidth=0.3,hatchsep=1.5,shadowsize=1,dimen=middle}
\psset{dotsize=0.7 2.5,dotscale=1 1,fillcolor=black}
\psset{arrowsize=1 2,arrowlength=1,arrowinset=0.25,tbarsize=0.7 5,bracketlength=0.15,rbracketlength=0.15}
\begin{pspicture}(0,0)(68.75,21.25)
\rput(5,11.25){$\pregr$}
\rput(5,21.25){$\xil$}
\rput(63.75,21.25){$\xir$}
\rput(67.5,5){$\rugt$}
\psline(8.75,7.5)(8.75,0)
\psline(1.25,7.5)(1.25,6.25)
\psline(1.25,7.5)(3.76,10)
\psline(6.26,10)(8.75,7.5)
\psline(5,13.75)(5,12.5)
\psline(3.76,12.5)(6.26,12.5)
\psline(3.76,10)(6.26,10)
\psline(6.26,12.5)(6.26,10)
\psline(3.76,12.5)(3.76,10)
\psline(2.5,6.25)(2.5,3.75)
\psline(0,6.25)(0,3.75)
\psline(0,6.25)(2.5,6.25)
\psline(0,3.75)(2.5,3.75)
\rput(1.25,5){$\lugt$}
\psline(6.25,16.25)(6.25,13.75)
\psline(3.75,13.75)(6.25,13.75)
\psline(3.75,16.25)(6.25,16.25)
\rput(5,15){$\prunt$}
\psline(3.75,16.25)(3.75,13.75)
\rput(63.75,11.25){$\pregr$}
\psline(60,7.5)(60,0)
\psline(67.5,7.5)(67.5,6.25)
\psline(60,7.5)(62.51,10)
\psline(65.01,10)(67.5,7.5)
\psline(63.75,13.75)(63.75,12.5)
\psline(62.51,12.5)(65.01,12.5)
\psline(62.51,10)(65.01,10)
\psline(65.01,12.5)(65.01,10)
\psline(62.51,12.5)(62.51,10)
\psline(68.75,6.25)(68.75,3.75)
\psline(66.25,6.25)(66.25,3.75)
\psline(66.25,6.25)(68.75,6.25)
\psline(66.25,3.75)(68.75,3.75)
\psline(65,16.25)(65,13.75)
\psline(62.5,13.75)(65,13.75)
\psline(62.5,16.25)(65,16.25)
\rput(63.75,15){$\prunt$}
\psline(62.5,16.25)(62.5,13.75)
\end{pspicture}

\end{split}
\eeq
The left-hand equivalence in \eqref{eq:radj-res} follows from the left-hand Frobenius condition in \eqref{eq:Frobspec-1}:
\beq\label{eq:residues-frob-1}
\begin{split}
\newcommand{\eex}{\scriptstyle x}
\newcommand{\eey}{\scriptstyle y}
\newcommand{\eev}{\scriptstyle z}
\newcommand{\prunt}{\scriptscriptstyle \spcun}
\newcommand{\lugt}{\scriptstyle \xi}
\newcommand{\fmnd}{\scriptscriptstyle \spmnd}
\newcommand{\fcmn}{\scriptscriptstyle \spcmn}
\def\JPicScale{1.1}
\ifx\JPicScale\undefined\def\JPicScale{1}\fi
\psset{unit=\JPicScale mm}
\psset{linewidth=0.3,dotsep=1,hatchwidth=0.3,hatchsep=1.5,shadowsize=1,dimen=middle}
\psset{dotsize=0.7 2.5,dotscale=1 1,fillcolor=black}
\psset{arrowsize=1 2,arrowlength=1,arrowinset=0.25,tbarsize=0.7 5,bracketlength=0.15,rbracketlength=0.15}
\begin{pspicture}(0,0)(63.75,18.75)
\psline(13.75,18.75)(13.75,11.25)
\psline(10,6.25)(10,0)
\psline(2.5,7.5)(2.5,6.25)
\psline(13.75,11.25)(11.25,8.75)
\psline(2.5,7.5)(5.01,10)
\psline(7.51,10)(8.75,8.75)
\psline(6.25,13.75)(6.25,12.5)
\rput[l](15,18.75){$\eev$}
\rput[r](9.38,0){$\eey$}
\psline(5.01,12.5)(7.51,12.5)
\psline(5.01,10)(7.51,10)
\psline(7.51,12.5)(7.51,10)
\psline(5.01,12.5)(5.01,10)
\psline(8.76,8.75)(11.26,8.75)
\psline(8.76,6.25)(11.26,6.25)
\psline(11.26,8.75)(11.26,6.25)
\psline(8.76,8.75)(8.76,6.25)
\psline(33.75,8.75)(33.75,6.25)
\psline(36.25,6.25)(38.75,3.75)
\psline(33.75,6.25)(31.25,3.75)
\psline(35,10)(35,8.75)
\psline(33.75,6.25)(36.25,6.25)
\psline(33.75,8.75)(36.25,8.75)
\psline(36.25,6.25)(36.25,8.75)
\psline(32.5,13.75)(33.75,12.5)
\psline(38.75,15)(36.25,12.5)
\psline(38.75,18.75)(38.75,15)
\psline(33.75,12.5)(33.75,10)
\psline(33.75,10)(36.25,10)
\psline(33.75,12.5)(36.25,12.5)
\psline(36.25,10)(36.25,12.5)
\psline(38.75,3.75)(38.75,0)
\rput[l](39.37,18.75){$\eev$}
\rput[r](38.13,0){$\eey$}
\rput(6.25,11.25){$\fmnd$}
\rput(35,11.25){$\fcmn$}
\rput(10,7.5){$\fcmn$}
\rput(35,7.5){$\fmnd$}
\psline(3.75,6.25)(3.75,3.75)
\psline(1.25,6.25)(1.25,3.75)
\psline(1.25,6.25)(3.75,6.25)
\psline(1.25,3.75)(3.75,3.75)
\rput(2.5,5){$\lugt$}
\psline(7.5,16.25)(7.5,13.75)
\psline(5,13.75)(7.5,13.75)
\psline(5,16.25)(7.5,16.25)
\rput(6.25,15){$\prunt$}
\psline(5,16.25)(5,13.75)
\psline(32.5,16.25)(32.5,13.75)
\psline(30,13.75)(32.5,13.75)
\psline(30,16.25)(32.5,16.25)
\rput(31.25,15){$\prunt$}
\psline(30,16.25)(30,13.75)
\psline(32.5,2.5)(32.5,0)
\psline(30,2.5)(30,0)
\psline(30,2.5)(32.5,2.5)
\psline(30,0)(32.5,0)
\rput(31.25,1.25){$\lugt$}
\psline(58.75,8.75)(58.75,6.25)
\psline(60,18.75)(60,8.75)
\psline(58.75,6.25)(61.25,6.25)
\psline(58.75,8.75)(61.25,8.75)
\psline(61.25,6.25)(61.25,8.75)
\rput[l](61.25,18.75){$\eev$}
\rput(60,7.5){$\fmnd$}
\rput(21.25,10){\EQLS}
\rput(48.75,10){\EQLS}
\psline(31.25,3.75)(31.25,2.5)
\psline(61.25,6.25)(63.75,3.75)
\psline(58.75,6.25)(56.25,3.75)
\psline(63.75,3.75)(63.75,0)
\rput[r](63.13,0){$\eey$}
\psline(57.5,2.5)(57.5,0)
\psline(55,2.5)(55,0)
\psline(55,2.5)(57.5,2.5)
\psline(55,0)(57.5,0)
\rput(56.25,1.25){$\lugt$}
\psline(56.25,3.75)(56.25,2.5)
\end{pspicture}

\end{split}
\eeq
The right-hand equivalence corresponds to the right-hand Frobenius condition. The equivalences in \eqref{eq:ladj-res} are $\ddag$-dual to \eqref{eq:radj-res}. For variation, we prove the converse by deriving the Frobenius condition from the left residual. Note that the defining conditions of the residuals say that the functions $\Rr,\Lr:\Up A\to \Do A$ and  $\Ll,\Rl:\Do A\to \Up A$ are surjective. For a given residuated monoid $(A,\mnd,\cmn,\unt,\cun)$, the left-hand of condition in \eqref{eq:ladj-res} thus implies:
\medskip
\beq\label{eq:residues-2}
\begin{split}
\newcommand{\fmnd}{\scriptstyle \therefore}
\newcommand{\fcmn}{\scriptstyle \because}
\newcommand{\funt}{\scriptstyle \unt}
\def\JPicScale{1.1}
\ifx\JPicScale\undefined\def\JPicScale{1}\fi
\psset{unit=\JPicScale mm}
\psset{linewidth=0.3,dotsep=1,hatchwidth=0.3,hatchsep=1.5,shadowsize=1,dimen=middle}
\psset{dotsize=0.7 2.5,dotscale=1 1,fillcolor=black}
\psset{arrowsize=1 2,arrowlength=1,arrowinset=0.25,tbarsize=0.7 5,bracketlength=0.15,rbracketlength=0.15}
\begin{pspicture}(0,0)(33.75,16.25)
\psline(0,16.25)(0,13.75)
\psline(3.75,8.75)(3.75,0)
\psline(0,13.75)(2.5,11.25)
\psline(5,11.25)(7.5,13.75)
\psline(7.5,16.25)(7.5,13.75)
\psline(2.5,11.25)(5,11.25)
\psline(2.5,8.75)(5,8.75)
\psline(5,11.25)(5,8.75)
\psline(2.5,11.25)(2.5,8.75)
\rput(3.75,10){$\fcmn$}
\psline(22.5,16.25)(22.5,11.25)
\psline(31.25,10)(33.75,7.5)
\psline(22.5,11.25)(25,8.75)
\psline(30,12.5)(30,16.25)
\psline(28.75,12.5)(31.25,12.5)
\psline(28.75,10)(31.25,10)
\psline(31.25,12.5)(31.25,10)
\psline(28.75,12.5)(28.75,10)
\psline(33.75,7.5)(33.75,0)
\psline(28.75,10)(27.5,8.75)
\psline(25,8.75)(27.5,8.75)
\psline(25,6.25)(27.5,6.25)
\rput(26.25,7.5){$\fcmn$}
\psline(27.5,8.75)(27.5,6.25)
\psline(25,8.75)(25,6.25)
\psline(26.25,6.25)(26.25,5)
\psline(25,5)(27.5,5)
\psline(25,2.5)(27.5,2.5)
\psline(27.5,5)(27.5,2.5)
\psline(25,5)(25,2.5)
\rput(30,11.25){$\fmnd$}
\rput(26.25,3.75){$\funt$}
\rput(15,8.75){\EQLS}
\end{pspicture}

\end{split}
\eeq
We use it at the first and at the last step, and in the middle the associativity:
\beq\label{eq:residues-frob-3}
\begin{split}
\newcommand{\funt}{\scriptstyle \unt}
\newcommand{\fmnd}{\scriptstyle \therefore}
\newcommand{\fcmn}{\scriptstyle \because}
\def\JPicScale{1.1}
\ifx\JPicScale\undefined\def\JPicScale{1}\fi
\psset{unit=\JPicScale mm}
\psset{linewidth=0.3,dotsep=1,hatchwidth=0.3,hatchsep=1.5,shadowsize=1,dimen=middle}
\psset{dotsize=0.7 2.5,dotscale=1 1,fillcolor=black}
\psset{arrowsize=1 2,arrowlength=1,arrowinset=0.25,tbarsize=0.7 5,bracketlength=0.15,rbracketlength=0.15}
\begin{pspicture}(0,0)(93.75,20)
\psline(88.75,8.75)(88.75,6.25)
\psline(91.25,6.25)(93.75,3.75)
\psline(88.75,6.25)(86.25,3.75)
\psline(90,11.25)(90,8.75)
\psline(88.75,6.25)(91.25,6.25)
\psline(88.75,8.75)(91.25,8.75)
\psline(91.25,6.25)(91.25,8.75)
\psline(86.25,16.25)(88.75,13.75)
\psline(93.75,16.25)(91.25,13.75)
\psline(86.25,20)(86.25,16.25)
\psline(93.75,20)(93.75,16.25)
\psline(88.75,13.75)(88.75,11.25)
\psline(88.75,11.25)(91.25,11.25)
\psline(88.75,13.75)(91.25,13.75)
\psline(91.25,11.25)(91.25,13.75)
\psline(86.25,3.75)(86.25,0)
\psline(93.75,3.75)(93.75,0)
\psline(0,20)(0,12.5)
\psline(5,6.25)(5,0)
\psline(15,7.5)(15,0)
\psline(0,12.5)(3.75,8.75)
\psline(6.25,8.75)(8.75,11.25)
\psline(11.25,11.25)(15,7.5)
\psline(10,20)(9.99,13.75)
\psline(8.75,13.75)(11.25,13.75)
\psline(8.75,11.25)(11.25,11.25)
\psline(11.25,13.75)(11.25,11.25)
\psline(8.75,13.75)(8.75,11.25)
\psline(3.75,8.75)(6.25,8.75)
\psline(3.75,6.25)(6.25,6.25)
\psline(6.25,8.75)(6.25,6.25)
\psline(3.75,8.75)(3.75,6.25)
\rput(5,7.5){$\fcmn$}
\rput(90,12.5){$\fcmn$}
\rput(90,7.5){$\fmnd$}
\rput(10,12.5){$\fmnd$}
\psline(26.25,20)(26.25,10)
\psline(36.25,7.5)(38.75,5)
\psline(43.75,7.5)(43.75,0)
\psline(26.25,10)(30,6.25)
\psline(36.25,10)(37.5,11.25)
\psline(40,11.25)(43.75,7.5)
\psline(38.75,20)(38.74,13.75)
\psline(37.5,13.75)(40,13.75)
\psline(37.5,11.25)(40,11.25)
\psline(40,13.75)(40,11.25)
\psline(37.5,13.75)(37.5,11.25)
\psline(33.75,10)(36.25,10)
\psline(33.75,7.5)(36.25,7.5)
\psline(36.25,10)(36.25,7.5)
\psline(33.75,10)(33.75,7.5)
\rput(38.75,12.5){$\fmnd$}
\psline(38.75,5)(38.75,0)
\psline(33.75,7.5)(32.5,6.25)
\psline(30,6.25)(32.5,6.25)
\psline(30,3.75)(32.5,3.75)
\rput(31.25,5){$\fcmn$}
\psline(32.5,6.25)(32.5,3.75)
\psline(30,6.25)(30,3.75)
\psline(31.25,3.75)(31.25,2.5)
\psline(30,2.5)(32.5,2.5)
\psline(30,0)(32.5,0)
\psline(32.5,2.5)(32.5,0)
\psline(30,2.5)(30,0)
\psline(55,20)(55,10)
\psline(72.5,7.5)(75,5)
\psline(55,10)(58.75,6.25)
\psline(70,10)(68.75,11.25)
\psline(67.5,20)(67.49,13.75)
\psline(66.25,13.75)(68.75,13.75)
\psline(66.25,11.25)(68.75,11.25)
\psline(68.75,13.75)(68.75,11.25)
\psline(66.25,13.75)(66.25,11.25)
\psline(70,10)(72.5,10)
\psline(70,7.5)(72.5,7.5)
\psline(72.5,10)(72.5,7.5)
\psline(70,10)(70,7.5)
\rput(67.5,12.5){$\fmnd$}
\psline(75,5)(75,0)
\psline(70,7.5)(67.5,5)
\psline(58.75,6.25)(61.25,6.25)
\psline(58.75,3.75)(61.25,3.75)
\rput(60,5){$\fcmn$}
\psline(61.25,6.25)(61.25,3.75)
\psline(58.75,6.25)(58.75,3.75)
\psline(60,3.75)(60,2.5)
\psline(58.75,2.5)(61.25,2.5)
\psline(58.75,0)(61.25,0)
\psline(61.25,2.5)(61.25,0)
\psline(58.75,2.5)(58.75,0)
\psline(66.25,11.25)(61.25,6.25)
\psline(67.5,5)(67.5,0)
\rput(35,8.75){$\fmnd$}
\rput(71.25,8.75){$\fmnd$}
\rput(60,1.25){$\funt$}
\rput(31.25,1.25){$\funt$}
\rput(20,10){\EQLS}
\rput(48.75,10){\EQLS}
\rput(80,10){\EQLS}
\end{pspicture}

\end{split}
\eeq
The second Frobenius law follows  symmetrically from the second residuation condition.
\end{proof}

\paragraph{\sf Examples.} Genuinely prelational, non-representable monoids $\MMM$ and $\LLL$ were described in Sec.~\ref{Sec:mon}. To get a residuated prelational monoid, expand the underlying set $\MMM$ of finite multisets of natural numbers $\NNn\to \NNn$ to the set $\ZZZ$ of finite multisets of integers, presented as finitely supported functions $\ZZz\to \NNn$, with the same preorder. The difference between the two examples provides an interesting view of the Frobenius decomposition of the sequents in the form $\derr {x,y} {u,v}$ into $\left({x,y}\leftarrow {x,s,v}\leftarrow {u,v}\right)$ and $\left({x,y}\leftarrow {u,t,y}\leftarrow {u,v}\right)$. The preorder $\widetilde \ZZZ$ of finite \emph{signed}\/ multisets of natural numbers, presented by finitely supported functions $\NNn\to \ZZz$, provides a slightly different situation, and another opportunity for an exercise in sequent algebra. The monoid $\LLL$ based on shuffling words, also described in Sec.~\ref{Sec:mon}, provides a non-example, as the residuation operations cannot be adjoined to it conservatively. 

\section{Spiders are unions of pregroups
}\label{Sec:unions}

\subsection{Unions}
A union $\bigcup \JJJ$ of a family of sets $\JJJ\subseteq \WP A$ is its supremum in the lattice of subsets $\WP A$. The concept of union extends naturally to a large class of categories \cite{BarrM:union}. It also extends from lattices of subsets to lattices of subalgebras. Lattice theory was originally developed to describe such extensions \cite{BirkhoffG:lattice}. If each element of $\JJJ$  carries some structure, say of a group, then their union will also carry that  structure, provided that they are \emph{consistent}\/ in a suitable sense. The consistency is usually equivalent to the requirement that any two structures in the given family are contained in a joint superstructure. The union is then the filtered supremum in the lattice of substructures. If a putative structure is a group, then the consistency requirement implies that all members of the family of groups $\JJJ$ share the same group unit. If they don't, then their union will not be a group, but a suitable generalization.  When the putative structure is not a group, but a pregroup, then the suitable generalization of the concept of pregroup is the concept of spider algebra. Our claim in this section is that the spider algebras of prelations are the unions of pregroups that are consistent, but not necessarily representable. Pregroups are monoids in the category $\Preord$ of preorders and monotone maps, but their unions are monoids in the category $\PRel$ of preorders and prelations, and the unions lift the property of being a pregroup into the property of being a spider. 

\smallskip
Just like the unions of sets are taken in a supremum completion $\WP A$, which is a lattice of subsets, the unions of preorders are taken in a supremum completion $\Do A$, which is the lattice of lower sets from Sec.~\ref{Sec:Completions}. For any family $\JJJ\subseteq \WP A$ of subpreorders of $A$ we define:
\bea
\bigcup \JJJ & = & \bigcup_{x \in X \in \JJJ} \downarrow x
\eea
To specify how $\bigcup \JJJ$ unifies the pregroup structures carried by the elements of $\JJJ$, we need to specify a notion of consistency. It seems simplest to specify it using pregroup representations.

\subsection{Cayley representation of monoids}
Any monoid $(A,\spmnd,\spunt)$ in any monoidal category $\CCC$ can be represented in the monoid of endomorphisms over its carrier along the embeddings $\ltimes, \rtimes \colon  \CCC(A)  \to \CCC(A,A)$, defined
\medskip
\beq\label{eq:cayley}
\begin{split}
\newcommand{\xil}{a^{\ltimes}(x) = a\cdot x}
\newcommand{\xir}{a^{\rtimes}(x) = x\cdot a}
\newcommand{\pregr}{\scriptstyle \therefore}
\newcommand{\lugt}{\scriptstyle a}
\def\JPicScale{1.1}
\ifx\JPicScale\undefined\def\JPicScale{1}\fi
\psset{unit=\JPicScale mm}
\psset{linewidth=0.3,dotsep=1,hatchwidth=0.3,hatchsep=1.5,shadowsize=1,dimen=middle}
\psset{dotsize=0.7 2.5,dotscale=1 1,fillcolor=black}
\psset{arrowsize=1 2,arrowlength=1,arrowinset=0.25,tbarsize=0.7 5,bracketlength=0.15,rbracketlength=0.15}
\begin{pspicture}(0,0)(48.75,21.25)
\rput(5,11.25){$\pregr$}
\rput(5,21.25){$\xil$}
\rput(43.75,21.25){$\xir$}
\psline(8.75,7.5)(8.75,0)
\psline(1.25,7.5)(1.25,6.25)
\psline(1.25,7.5)(3.76,10)
\psline(6.26,10)(8.75,7.5)
\psline(5,17.5)(5,12.5)
\psline(3.76,12.5)(6.26,12.5)
\psline(3.76,10)(6.26,10)
\psline(6.26,12.5)(6.26,10)
\psline(3.76,12.5)(3.76,10)
\psline(2.5,6.25)(2.5,3.75)
\psline(0,6.25)(0,3.75)
\psline(0,6.25)(2.5,6.25)
\psline(0,3.75)(2.5,3.75)
\rput(1.25,5){$\lugt$}
\rput(43.75,11.25){$\pregr$}
\psline(40,7.5)(40,0)
\psline(47.5,7.5)(47.5,6.25)
\psline(40,7.5)(42.51,10)
\psline(45.01,10)(47.5,7.5)
\psline(43.75,17.5)(43.75,12.5)
\psline(42.51,12.5)(45.01,12.5)
\psline(42.51,10)(45.01,10)
\psline(45.01,12.5)(45.01,10)
\psline(42.51,12.5)(42.51,10)
\psline(48.75,6.25)(48.75,3.75)
\psline(46.25,6.25)(46.25,3.75)
\psline(46.25,6.25)(48.75,6.25)
\psline(46.25,3.75)(48.75,3.75)
\rput(47.5,5){$\lugt$}
\end{pspicture}

\end{split}
\eeq
where the elements $\CCC(A)$ are the morphisms from the monoidal unit. The embeddings in \eqref{eq:cayley} are the monoid part of Cayley's group representation, and also a special case of the Yoneda embeddings. An endomorphism $f\in \CCC(A,A)$ is a representative in the image of $\rtimes$ if and only if $f(x\cdot y) = \scomp {f(x)} y$ and in the image of $\ltimes$ if and only if $f(x\cdot y) = \scomp y {f(x)}$. An inverse image, \emph{viz}\/ a represented element, can be obtained as $f(\spunt)$ in both cases.  The associativity of the monoid assures that $(a\cdot b)^{\rtimes} = \scomp {a^\ltimes} {b^\ltimes}$, $(a\cdot b)^{\ltimes} = \scomp {b^\ltimes} {a^\ltimes}$,  and $\spunt^\ltimes = \spunt^\rtimes = \id$ make both embeddings into monoid homomorphisms. When $\CCC$ is enriched, the monoid enrichment is preserved by definition.

\smallskip
How about the duality? When $\CCC$ has an identity-on-the-objects monoidal duality $\ddag:\CCC^{op}\to \CCC$, then the Frobenius conditions assure that the monoid operations $L, R: \CCC(A)\to \CCC(A)$, defined
\medskip
\beq\label{eq:star}
\begin{split}
\newcommand{\xil}{a^L = \left(\spunt \, ; \spcmn\, ;(a^\ddag \times A) \right)}
\newcommand{\xir}{a^R = \left(\spunt \, ; \spcmn\, ;(A\times a^\ddag) \right)}
\newcommand{\pregr}{\scriptstyle \because}
\newcommand{\lugt}{\raisebox{.25ex}{$\scriptscriptstyle a^\ddag$}}
\newcommand{\srce}{\scriptstyle \spunt}
\def\JPicScale{1.1}
\ifx\JPicScale\undefined\def\JPicScale{1}\fi
\psset{unit=\JPicScale mm}
\psset{linewidth=0.3,dotsep=1,hatchwidth=0.3,hatchsep=1.5,shadowsize=1,dimen=middle}
\psset{dotsize=0.7 2.5,dotscale=1 1,fillcolor=black}
\psset{arrowsize=1 2,arrowlength=1,arrowinset=0.25,tbarsize=0.7 5,bracketlength=0.15,rbracketlength=0.15}
\begin{pspicture}(0,0)(50,21.25)
\rput(5,5){$\pregr$}
\rput(5,21.25){$\xil$}
\rput(43.75,21.25){$\xir$}
\psline(8.75,16.25)(8.75,8.75)
\psline(6.25,6.25)(8.76,8.75)
\psline(1.26,8.75)(3.75,6.25)
\psline(5,3.75)(5,2.5)
\psline(3.76,6.25)(6.26,6.25)
\psline(3.76,3.75)(6.26,3.75)
\psline(6.26,6.25)(6.26,3.75)
\psline(3.76,6.25)(3.76,3.75)
\psline(2.5,12.5)(2.5,10)
\psline(0,12.5)(0,10)
\psline(0,12.5)(2.5,12.5)
\psline(0,10)(2.5,10)
\rput(1.25,11.25){$\lugt$}
\psline(1.25,10)(1.25,8.75)
\psline(6.25,2.5)(6.25,0)
\psline(3.75,2.5)(3.75,0)
\psline(3.75,2.5)(6.25,2.5)
\psline(3.75,0)(6.25,0)
\rput(5,1.25){$\srce$}
\rput(45,5){$\pregr$}
\psline(41.25,16.25)(41.25,8.75)
\psline(46.25,6.25)(48.76,8.75)
\psline(41.26,8.75)(43.75,6.25)
\psline(45,3.75)(45,2.5)
\psline(43.76,6.25)(46.26,6.25)
\psline(43.76,3.75)(46.26,3.75)
\psline(46.26,6.25)(46.26,3.75)
\psline(43.76,6.25)(43.76,3.75)
\psline(50,12.5)(50,10)
\psline(47.5,12.5)(47.5,10)
\psline(47.5,12.5)(50,12.5)
\psline(47.5,10)(50,10)
\rput(48.75,11.25){$\lugt$}
\psline(48.75,10)(48.75,8.75)
\psline(46.25,2.5)(46.25,0)
\psline(43.75,2.5)(43.75,0)
\psline(43.75,2.5)(46.25,2.5)
\psline(43.75,0)(46.25,0)
\rput(45,1.25){$\srce$}
\end{pspicture}

\end{split}
\eeq
induce dualities on the monoid $A$, and that these dualities are preserved under the representations, in the sense 
\beq
\left(a^L\right)^\ltimes = \left(a^\ltimes\right)^\ddag \qquad \qquad \qquad \qquad \left(a^R\right)^\rtimes = \left(a^\rtimes\right)^\ddag
\eeq
This is proved by applying a Frobenius condition on each side. The converse, that such dualities also imply the Frobenius conditions, is in \cite[Thm.~4.3]{PavlovicD:Qabs12}. The embeddings of self-dual monoids into the monoids of endomorphisms in self-dual categories were further discussed and exploited in \cite{PavlovicD:MSCS13,PavlovicD:Qabs12}. 

\smallskip

Instantiating to $\CCC = \PRel$, the embeddings become $\ltimes, \rtimes \colon  \Do A  \to \Do(A^o\times A)$, since  $\PRel(A) = \Do A$ and $\PRel(A,A) = \Do(A^o\times A)$. If we are given a prelational monoid $(A,\spmnd,\spunt)$, the monoid unit $\spunt \in \Do A$ thus induces representations $\spunt^{\ltimes}, \spunt^{\rtimes}\in  \Do(A^o\times A)$ which both correspond to the identity prelation $\id_{A}$, which is just the preordering of $A$
\[
\derrr x {\spunt^{\ltimes}} y \ \ \iff \ \ \derrr x A y \ \ \iff \ \ \derrr x{\spunt^{\rtimes}} y
\]
It is easy to see that any lower set $\upsilon \subseteq \spunt$ induces representations $\upsilon^\ltimes, \upsilon^\rtimes \subseteq \id_A$ which determine a subpreorder $A_\upsilon \subseteq A$ such that $\upsilon^\ltimes=\upsilon^\rtimes = \id_{A_\upsilon}$.

\subsection{Consistency of spiders}
If we are not given  a prelational monoid on $A$, but two prelational monoids $(A_{0},\spmnd_{0},\spunt_{0})$ and $(A_{1},\spmnd_{1},\spunt_{1})$ for $A_{0}, A_{1} \subseteq A$ then $\spunt_{0}^{\ltimes_{0}} = \id_{A_{0}}$ and $\spunt_{1}^{\ltimes_{1}} = \id_{A_1}$. We say that the monoids $A_{0}$ and $A_{1}$ are \emph{consistent}\/ if for every $a\in A_{0}\cap A_{1}$ holds $a^{\ltimes_{0}} = a^{\ltimes_{1}}$.

\smallskip
A \emph{monoid covering}\/ of a preorder $A$ is
\begin{itemize}
\item a family of preorders $\JJJ \subseteq \Do A$, written in the form $\JJJ = \{A_{i}\}_{i\in J}$, such that $\bigcup \JJJ = A$,
\item a family of monoids $(A_{i}, \spmnd_i, \spunt_i)$ such that for all $i,j\in I$ the monoids $A_i$ and $A_j$ consistent, i.e. 
\bea
\derrr {a,x} {\spmnd_i} y & \iff & \derrr {a,x} {\spmnd_j} y
\eea
holds for every $a\in A_i\cap A_j$ and all $x,y\in A$.
\end{itemize}
 
%
%

\begin{theorem}\label{thm:three}
A preorder is a spider if and only if it has a pregroup covering.
\end{theorem} 

\begin{proof}
For any monoid $(A,\spmnd,\spunt)$ in $\PRel$, the representative $\spunt^{\ltimes}:A\prel A$ is the identity. Viewing $\spunt:\OOne\prel A$ as the lower set $\spunt\in \Do A$, every $t\in \spunt$ also induces a prelation $t:\OOne \prel A$ corresponding to the principal lower set $\downarrow t\in \Do A$. The representative $t^{\ltimes}:A\prel A$ satisfies $t^{\ltimes }\subseteq \spunt^{\ltimes}=\id_{A}$. There is thus $A_{t}\subseteq A$ such that $t^{\ltimes } = \id_{A_{t}}$. By definition $x \in A$ is in $A_t$ if and only if $x^\ltimes = \scomp {x^\ltimes} {t^\ltimes} = \scomp {t^\ltimes} {x^\ltimes}$. For $x,y\in A_t$ it follows that 
\beq \label{eq:closed}
\scomp {x^\ltimes} {y^\ltimes} \ = \ \left( {t^\ltimes} \, ; {x^\ltimes}\, ; {y^\ltimes}\right)\  =\  \left( {x^\ltimes} \, ; {t^\ltimes}\, ; {y^\ltimes}\right)\  =\    \left( {x^\ltimes} \, ; {y^\ltimes}\, ; {t^\ltimes}\right)
\eeq
which means that $A_t$ is closed under the monoid operation in $A$. Denoting the restriction of $\spmnd$ to $A_t$ by $\spmnd_{t}$ yields the submonoid $(A_t,\spmnd_t, t)$ of $(A,\spmnd,\spunt)$. If the latter monoid is a spider, then it has the dualities from \eqref{eq:star}, which are inherited by $A_t$ by reasoning similar to \eqref{eq:closed}. It follows that the monoid $A_t$ is also a spider, since conditions (\ref{eq:Frobspec-1}--\ref{eq:Frobspec-2}) follow from the dualities. But  $A_{t}$ is a pointed monoid, with the unit represented by $t$. As a pointed spider, $A_{t}$ is thus a pregroup, by Thm.~\ref{thm:one}. Since $t$ was taken as an arbitrary element of $\spunt$, and $\id_A = \spunt^{\ltimes} = \bigcup_{t\in \spunt} t^{\ltimes} = \bigcup_{t\in \spunt} \id_{A_t}$, it follows that
\bear
A & = & \bigcup_{t\in \spunt} A_{t}
\eear
where all $A_{t}$ are pregroups. The consistency follows from the fact that for all $t\in \spunt$, the monoid operations $\spmnd_t$ on $A_t$ are restrictions of $\spmnd$ on $A$, and hence $x^{\ltimes_t}= x^{\ltimes}$.  We have thus constructed a pregroup cover of an arbitrary spider.

\smallskip
The other way around, let $\JJJ = \{(A_i,\mnd_i, \unt_i)\}_{i\in J}$ be a consistent family of pregroups over subpreorders of $A=\bigcup_{i\in J} A_i$. For representable monoids, the consistency boils down to the requirement that all $a,x \in A_i \cap A_j$ satisfy\footnote{We lift the notation $a\cdot x$ for the representive $\derrr {a\cdot x}{A} y$ of $\derrr{a,x}{\mnd} y$ from $A$ to $A_i$.} $a\cdot_i x = a \cdot_j x$. The consistency thus implies that $\mnd = \bigcup_{i\in J} \mnd_i$, together with $\unt = \bigcup_{i \in J}$ yields a monoid structure $(A,\mnd,\unt)$, since the monoid conditions lift from the components to the union. It remains to be proved that the monoid  is a spider. Thm.~\ref{thm:one} assured that each of the pregroups $A_i$ is a spider. The fact that the consistency assures that the spider conditions also lift from the components to the union can also be checked directly, but less directly than the monoid conditions, because it involves the dual comonoids. A shorter path is to show that the consistency of the monoids implies the consistency of the duality operations $L_i,R_i$ from \eqref{eq:star} for all $i\in J$, which is straightforward. The union then produces the duality operations $L,R$ on the monoid $(A,\mnd,\unt)$, which are equivalent to the Frobenius conditions by \cite[Thm.~4.3]{PavlovicD:Qabs12}, and the specialty condition by  \cite[Lemma.~4.8]{PavlovicD:Qabs12}. This completes the construction of the spider $A$ as a union of a consistent family of pregroups $\JJJ$.
\end{proof}

\paragraph{\sf Pregroupoids?} Restricted to \emph{symmetric}\/ preorders, the unions in Thm.~\ref{thm:three} become disjoint, pregroups boil down to groups, and the result essentially boils down to \cite{PavlovicD:QI09}, where the abelian spiders in the category $\Rel$ of sets and relations were characterized as the disjoint unions of abelian groups. That result has at times been restated in terms of groupoids instead of the disjoint unions of groups. Formally, the two frameworks are equivalent, in the sense that any disjoint union of groups is a groupoid, and the skeleton of any groupoid is a disjoint union of groups, equivalent to it. In practice, however, groupoids are often significantly more flexible than the disjoint unions of groups. In homotopy theory, this was compellingly argued a while ago  \cite{BrownR:groupoids,HigginsPJ:book,Esquisse}. Reducing fundamental groupoids to fundamental groups, or splitting the quotient functors from categories to their skeletons in general, requires choosing base points, which can be arbitrarily hard. It is therefore important to note that the decompositions of spider algebras into disjoint unions of groups  in $\Rel$ \cite{PavlovicD:QI09}, and into unions of pregroups  in $\PRel$ here, \emph{come with canonical base points}, which are given by the units as the canonical base points. Omitting the base points and stating the results in terms of groupoids substantially weakens the original claims. The descriptions in terms of unions are less flexible precisely because they are more informative. On the other hand, the flexibility gained by weakening the described spider decompositions allows capturing a different correspondence, aligning the Frobenius law with invertible partial compositions \cite{HeunenC:frob-groupoid}. Lifting this idea to the current framework suggests an interesting generalization of the pregroup unions to suitable fibrations, which must be left for future work.  
%
%
%
%

\section{Summary and further work}\label{Sec:outro}

\subsubsection*{[Allergy Warning: The following paragraphs may contain traces of philosophy.]}

This is yet another paper about adjunctions. Lawvere showed that logical operations arise from adjunctions  \cite{LawvereFW:dialectica}. The reason is that the logical introduction and elimination rules can be specified as sequent adjunctions. Lambek showed that substantial fragments of the syntax of natural languages can be reconstructed in terms of adjunctions \cite{LambekJ:word-sentence}. The reason is that the sequents used to analyze syntax can be reduced to contractions along the counits of adjoint pairs of syntactic types.\footnote{Since the adjunction units are never used in syntactic analyses, the question whether the actual adjunctions are really needed has been raised from the outset. In \cite{LambekJ:grammar97}, pregroups were introduced only after an extended discussion about \emph{"protogroups"}, which are partially ordered monoids with left and right "protoinverses" $x^\ell$ and $x^r$ equipped with contractions $\derr \unt {x^\ell x}$ and $\derr \unt {xx^r}$, but no expansions that would make them adjoint to $x$. The structure of a protogroup without the equation $x^{\ell r} = x^{r\ell}=x$ imposed is not intrinsic to a monoid in the sense of Sec~\ref{Sec:property}. Adding such protoinverses to a monoid obliterates any protoinverses, or any adjoints, that might have existed in it previously, and proliferates spurious types and contractions. On the other hand, a protogroup satisfying the equation $x^{\ell r} = x^{r\ell}$ is a pregroup, with $\derr {x x^\ell} \unt  = \ell\derr\unt {xx^r} $ and $\derr {x^r x} \unt  = r\derr  \unt{x^\ell x}$. Maybe the observation that only the adjunction counits are used is in the eye of the observer? Maybe we use the units without noticing? Maybe the symmetry of the adjunction structure echoes the symmetry of communication: the counits are used to analyze sentences, the units to synthesize them?} Ongoing work suggests that data analysis, concept mining, and even communication in general evolve as adjunctions \cite{PavlovicD:SemSp21,PavlovicD:CALCO15,PavlovicD:ICFCA12,PavlovicD:nucleus}. What is the impact of the presented results on these observations? 

\smallskip
The conceptual impact is that there is a sequent-algebraic connection between
\begin{enumerate}[a)]
\item the pregroup reductions ${x x^\ell} \leftarrow\, \unt \leftarrow {x^\ell  x}$ and  ${x^r  x} \leftarrow\, \unt \leftarrow {x x^r}$ and
\item the spider decompositions of $\derr {xy} {uv}$ into $\left({xy}\leftarrow {xsv}\leftarrow {uv}\right)$ and  $\left({xy}\leftarrow {uty}\leftarrow {uv}\right)$. 
\end{enumerate}
The syntactic processes depicted in (a) open and close the brackets of meaning, as noted by Husserl \cite[Logical Investigation IV]{HusserlE:shorter} and formalized by Ajdukiewicz \cite{Ajdukiewicz}. The formalism distinguished $x^\ell$, as a charge of a later reference $x$, from $x^r$, as a discharge of an earlier $x$. Closing such "connexions" by reducing them to $\unt$ set the well-formed sentences apart from the background. Processes (b) allow factoring general syntactic reductions into single-component reductions. The algebraic equivalence of the two different presentations of the syntactic process does seem to have repercussions on the original epistemological analyses of language, and possibly on the  present computational applications, but they are far from clear at present.

\smallskip
The technical impact of the presented results is clearer, but they only provide an algebraic stepping stone into a problem area that opens in many directions. In \cite[Ch.~28]{LambekJ:word-sentence}, Lambek discussed the limitations of pregroups as a context-insensitive model, and the paths forward. In the meantime, the advances in the pragmatics of natural language processing by neural nets broadened the scope of the concept of context. A context is not a mere interface between syntax and semantics. In the architecture of fast-learning nets, a context is more fruitfully interpreted as a \emph{process of selecting a syntax suitable for a given source}. The narratives are parsed at several levels, with many layers of syntax. The grammars applicable at lower levels are determined by higher-level contexts. The source streams the contextual information. This communication architecture is echoed in the presented sequent-algebra model as the incremental narrowing of the monoid unit from sequent to sequent, selecting within the lower set that plays the role of the spider-algebra unit an element that plays the role of a pregroup unit and enables a syntactic reduction of the given phrase. This informal description of the process is offered as an indication of the research tasks ahead.

\smallskip
Leaving aside the presented results and their interpretations, the alignment of the general framework of prelational sequent algebra with the DisCoCat program seems to suggest a possible link between the functorial vector space semantics  \cite{CoeckeB:APAL13} and the task of understanding the monoidal functorial semantics in general \cite{PavlovicD:FunSem17}. If the characterization of the relational spider algebras as groups broadened the view of the basis sets beyond the linear spaces, uncovering groups as resources within the basis elements, then the characterization of the perordered spider algebras as pregroups refines the view of that resource beyond groups. Initially, the spider algebras over vector spaces provided an element-free view of the \emph{linear decompositions}\/ over the basis sets \cite{PavlovicD:MSCS13}. The relational and the prelational algebras as presented also provide access to the \emph{dynamic compositions}\/ within the basis elements. The comonoid part of a spider provides the external decompositions, the monoid part provides the internal compositions, and the Frobenius condition provides the interface between the two. It keeps the spider together.

\bibliographystyle{plainurl}
\bibliography{CT,language,logic,semantics,PavlovicD,ref-spidspid,coecke,math}

\begin{thebibliography}{10}

\bibitem{Ajdukiewicz}
Kazimierz Ajdukiewicz.
\newblock Die syntaktische konnexit\"at.
\newblock {\em Studia Philosophica}, 1:1--27, 1935.
\newblock Translation in: \emph{Polish Logic 1920--1939} (S.~McCall, ed),
  Oxford University Press 1967, pp 202--231.

\bibitem{Bar-HillelY:categorial}
Yehoshua Bar-Hillel.
\newblock A quasi-arithmetical notation for syntactic description.
\newblock {\em Language}, 29(1):47--58, 1953.
\newblock \href {https://doi.org/10.2307/410452} {\path{doi:10.2307/410452}}.

\bibitem{BarrM:union}
Michael Barr.
\newblock On categories with effective unions.
\newblock In F.~Bourceux, editor, {\em Categorical algebra and its
  applications}, volume 1348 of {\em Lecture Notes in Mathematics}, pages
  19--35. Springer-Verlag, 1988.
\newblock \href {https://doi.org/10.1007/bfb0081346}
  {\path{doi:10.1007/bfb0081346}}.

\bibitem{BirkhoffG:lattice}
Garrett Birkhoff.
\newblock {\em Lattice Theory}, volume~25 of {\em American Mathematical Society
  Colloquium Publications}.
\newblock American Mathematical Society, 1940.
\newblock \href {https://doi.org/10.1090/coll/025}
  {\path{doi:10.1090/coll/025}}.

\bibitem{BloomfieldL:Panini}
Leonard Bloomfield.
\newblock On some rules of {Panini}.
\newblock {\em Journal of the American Oriental Society}, 47:61--70, 1927.
\newblock URL: \url{http://www.jstor.org/stable/593241}.

\bibitem{BonchiF:LICS18}
Filippo Bonchi, Fabio Gadducci, Aleks Kissinger, Pawel Sobocinski, and Fabio
  Zanasi.
\newblock Rewriting with frobenius.
\newblock In Anuj Dawar and Erich Gr{\"{a}}del, editors, {\em Proceedings of
  the 33rd Annual {ACM/IEEE} Symposium on Logic in Computer Science, {LICS}
  2018, Oxford, UK, July 09-12, 2018}, pages 165--174. {ACM}, 2018.
\newblock \href {https://doi.org/10.1145/3209108.3209137}
  {\path{doi:10.1145/3209108.3209137}}.

\bibitem{PavlovicD:FunSem17}
Filippo Bonchi, Dusko Pavlovic, and {Pawe\l} Soboci\'nski.
\newblock {Functorial Semantics for Relational Theories}.
\newblock Technical report, ASECOLab, November 2017.
\newblock \href {https://doi.org/https://doi.org/10.48550/arXiv.1711.08699}
  {\path{doi:https://doi.org/10.48550/arXiv.1711.08699}}.

\bibitem{LewisM:transl}
Tai{-}Danae Bradley, Martha Lewis, Jade Master, and Brad Theilman.
\newblock Translating and evolving: Towards a model of language change in
  {DisCoCat}.
\newblock In M.~Lewis et~al, editor, {\em Proceedings of the CAPNS@QI 2018},
  volume 283 of {\em {EPTCS}}, pages 50--61, 2018.
\newblock \href {https://doi.org/10.4204/EPTCS.283.4}
  {\path{doi:10.4204/EPTCS.283.4}}.

\bibitem{BrownR:groupoids}
Ronald Brown.
\newblock {\em Topology and Groupoids}.
\newblock BookSurge Publishing, 3 edition.
\newblock URL: \url{https://groupoids.org.uk/topgpds.html}.

\bibitem{GPT3}
Tom~B. Brown, Benjamin Mann, Nick Ryder, Melanie Subbiah, Jared Kaplan,
  Prafulla Dhariwal, Arvind Neelakantan, Pranav Shyam, Girish Sastry, Amanda
  Askell, Sandhini Agarwal, Ariel Herbert{-}Voss, Gretchen Krueger, Tom
  Henighan, Rewon Child, Aditya Ramesh, Daniel~M. Ziegler, Jeffrey Wu, Clemens
  Winter, Christopher Hesse, Mark Chen, Eric Sigler, Mateusz Litwin, Scott
  Gray, Benjamin Chess, Jack Clark, Christopher Berner, Sam McCandlish, Alec
  Radford, Ilya Sutskever, and Dario Amodei.
\newblock Language models are few-shot learners.
\newblock {\em CoRR}, abs/2005.14165, 2020.
\newblock \href {http://arxiv.org/abs/2005.14165} {\path{arXiv:2005.14165}},
  \href {https://doi.org/10.48550/arXiv.2005.14165}
  {\path{doi:10.48550/arXiv.2005.14165}}.

\bibitem{BuszkowskiW:synthese07}
Wojciech Buszkowski.
\newblock Type logics and pregroups.
\newblock {\em Studia Logica}, 87(2-3):145--169, 2007.
\newblock \href {https://doi.org/10.1007/s11225-007-9083-4}
  {\path{doi:10.1007/s11225-007-9083-4}}.

\bibitem{Buszkowski-Moroz:CFL}
Wojciech Buszkowski, Zhe Lin, and Katarzyna Moroz.
\newblock Pregroup grammars with letter promotions: Complexity and
  context-freeness.
\newblock {\em J. Comput. Syst. Sci.}, 78(6):1899--1909, 2012.
\newblock \href {https://doi.org/10.1016/j.jcss.2011.12.010}
  {\path{doi:10.1016/j.jcss.2011.12.010}}.

\bibitem{CarboniA:matrices}
Aurelio Carboni.
\newblock Matrices, relations, and group representations.
\newblock {\em Journal of Algebra}, 136(2):497--529, 1991.
\newblock \href {https://doi.org/10.1016/0021-8693(91)90057-f}
  {\path{doi:10.1016/0021-8693(91)90057-f}}.

\bibitem{Carboni-Walters}
Aurelio Carboni and Robert~F.C. Walters.
\newblock Cartesian bicategories, {I}.
\newblock {\em J. of Pure and Applied Algebra}, 49:11--32, 1987.
\newblock \href {https://doi.org/10.1016/0022-4049(87)90121-6}
  {\path{doi:10.1016/0022-4049(87)90121-6}}.

\bibitem{Lambek-Casadio:four}
Claudia Casadio and Joachim Lambek.
\newblock A tale of four grammars.
\newblock {\em Studia Logica}, 71(3):315--329, 2002.
\newblock \href {https://doi.org/10.1023/A:1020564714107}
  {\path{doi:10.1023/A:1020564714107}}.

\bibitem{ChomskyN:three}
N.~Chomsky.
\newblock Three models for the description of language.
\newblock {\em IRE Transactions on Information Theory}, 2(3):113--124, 1956.
\newblock \href {https://doi.org/10.1109/TIT.1956.1056813}
  {\path{doi:10.1109/TIT.1956.1056813}}.

\bibitem{ChomskyN:structures}
Noam Chomsky.
\newblock {\em Syntactic Structures}.
\newblock Mouton, The Hague, 1957.
\newblock \href {https://doi.org/10.1515/9783110218329}
  {\path{doi:10.1515/9783110218329}}.

\bibitem{Coecke-Martha-Marsden}
Bob Coecke, Fabrizio Genovese, Martha Lewis, Dan Marsden, and Alexis Toumi.
\newblock Generalized relations in linguistics {\&} cognition.
\newblock {\em Theor. Comput. Sci.}, 752:104--115, 2018.
\newblock \href {https://doi.org/10.1016/j.tcs.2018.03.008}
  {\path{doi:10.1016/j.tcs.2018.03.008}}.

\bibitem{CoeckeB:APAL13}
Bob Coecke, Edward Grefenstette, and Mehrnoosh Sadrzadeh.
\newblock {L}ambek vs. {L}ambek: Functorial vector space semantics and string
  diagrams for {L}ambek calculus.
\newblock {\em Ann. Pure Appl. Log.}, 164(11):1079--1100, 2013.
\newblock \href {https://doi.org/10.1016/j.apal.2013.05.009}
  {\path{doi:10.1016/j.apal.2013.05.009}}.

\bibitem{CoeckeB:Picturing}
Bob Coecke and Aleks Kissinger.
\newblock {\em Picturing Quantum Processes: {A} First Course in Quantum Theory
  and Diagrammatic Reasoning}.
\newblock Cambridge University Press, 2017.
\newblock \href {https://doi.org/10.1017/9781316219317}
  {\path{doi:10.1017/9781316219317}}.

\bibitem{PavlovicD:CQStruct}
Bob Coecke, Éric~Oliver Paquette, and Dusko Pavlovic.
\newblock Classical and quantum structuralism.
\newblock In Simon Gay and IanEditors Mackie, editors, {\em Semantic Techniques
  in Quantum Computation}, page 29–69. Cambridge University Press, 2009.
\newblock \href {https://doi.org/10.1017/CBO9781139193313.003}
  {\path{doi:10.1017/CBO9781139193313.003}}.

\bibitem{PavlovicD:QMWS}
Bob Coecke and Dusko Pavlovic.
\newblock Quantum measurements without sums.
\newblock In G.~Chen, L.~Kauffman, and S.~Lamonaco, editors, {\em Mathematics
  of Quantum Computing and Technology}, page 36pp. Taylor and Francis, 2007.

\bibitem{PavlovicD:MSCS13}
Bob Coecke, Dusko Pavlovic, and Jamie Vicary.
\newblock A new description of orthogonal bases.
\newblock {\em Math. Structures in Comp. Sci.}, 23(3):555--567, 2013.
\newblock \href {https://doi.org/10.1017/S0960129512000047}
  {\path{doi:10.1017/S0960129512000047}}.

\bibitem{CoeckeB:discocat}
Bob Coecke, Mehrnoosh Sadrzadeh, and Stephen Clark.
\newblock Mathematical foundations for a compositional distributional model of
  meaning.
\newblock In J.~van Benthem, M.~Moortgat, and W.~Buszkowski, editors, {\em A
  Festschrift for Jim Lambek}, Linguistic Analysis, pages 345--384. Springer,
  2010.

\bibitem{SobocinskiP:POPL21}
Ivan {Di Liberti}, Fosco Loregi{\`{a}}n, Chad Nester, and Pawel Sobocinski.
\newblock Functorial semantics for partial theories.
\newblock {\em Proc. {ACM} Program. Lang.}, 5({POPL}):1--28, 2021.
\newblock \href {https://doi.org/10.1145/3434338} {\path{doi:10.1145/3434338}}.

\bibitem{GentzenG:seq}
Gerhard Gentzen.
\newblock {Untersuchungen {\"u}ber das logische Schlie{\ss}en}.
\newblock {\em Mathematische Zeitschrift}, 39(1):176--210, 405--431, 1935.

\bibitem{HeunenC:frob-groupoid}
Chris Heunen, Ivan Contreras, and Alberto~S. Cattaneo.
\newblock Relative {Frobenius} algebras are groupoids.
\newblock {\em Journal of Pure and Applied Algebra}, 217(1):114--124, 2013.
\newblock \href {https://doi.org/10.1016/j.jpaa.2012.04.002}
  {\path{doi:10.1016/j.jpaa.2012.04.002}}.

\bibitem{Heunen-Sadrzadeh}
Chris Heunen, Mehrnoosh Sadrzadeh, and Edward Grefenstette, editors.
\newblock {\em Quantum Physics and Linguistics - {A} Compositional,
  Diagrammatic Discourse}.
\newblock Oxford University Press, 2013.
\newblock \href {https://doi.org/10.1093/acprof:oso/9780199646296.001.0001}
  {\path{doi:10.1093/acprof:oso/9780199646296.001.0001}}.

\bibitem{HigginsPJ:book}
Philip~J. Higgins.
\newblock {\em Categories and groupoids}.
\newblock Van Nostrand Reinhold, 1971.
\newblock also in \emph{Reprints of Theory and Applications of Categories}, No.
  7 (2005) pp 1-195.

\bibitem{PavlovicD:SemSp21}
Dominic~J.D. Hughes and Dusko Pavlovic.
\newblock {Sign as an adjunction}.
\newblock working paper.

\bibitem{HusserlE:shorter}
Edmund Husserl.
\newblock {\em The Shorter Logical Investigations}.
\newblock International Library of Philosophy. Routledge, 2001.
\newblock \href {https://doi.org/10.4324/9780203420034}
  {\path{doi:10.4324/9780203420034}}.

\bibitem{JurafskyD:book-2020}
Daniel Jurafsky and James~H. Martin.
\newblock Speech and language processing, December 2020.
\newblock partial manuscript of 3rd edition, revised and expanded, available
  through authors' home page.

\bibitem{PavlovicD:CALCO15}
Toshiki Kataoka and Dusko Pavlovic.
\newblock {Towards Concept Analysis in Categories: Limit Inferior as Algebra,
  Limit Superior as Coalgebra}.
\newblock In L.S. Moss and P.~Sobocinski, editors, {\em Proceedings of CALCO
  2015}, volume~35 of {\em LIPIcs}, pages 130--155, Dagstuhl, Germany, 2015.
  Leibniz-Zentrum f\"ur Informatik.
\newblock \href {https://doi.org/10.48550/arXiv.1505.01098}
  {\path{doi:10.48550/arXiv.1505.01098}}.

\bibitem{KockJ:frobenius}
Joachim Kock.
\newblock {\em Frobenius Algebras and 2-D Topological Quantum Field Theories},
  volume~59 of {\em Londong Mathematical Society Student Texts}.
\newblock Cambridge University Press, 2004.
\newblock \href {https://doi.org/10.1017/cbo9780511615443}
  {\path{doi:10.1017/cbo9780511615443}}.

\bibitem{LambekJ:pregroups}
Jim Lambek.
\newblock Pregroups: {A} new algebraic approach to sentence structure.
\newblock In Carlos Mart{\'{\i}}n{-}Vide and Gheorghe P{a}un, editors, {\em
  Recent Topics in Mathematical and Computational Linguistics}, pages 182--195.
  Editura Academiei Rom{\^{a}}ne, Bucuresti, 2000.

\bibitem{LambekJ:math-sentence}
Joachim Lambek.
\newblock The mathematics of sentence structure.
\newblock {\em The American Mathematical Monthly}, 65(3):154--170, 1958.
\newblock \href {https://doi.org/10.1080/00029890.1958.11989160}
  {\path{doi:10.1080/00029890.1958.11989160}}.

\bibitem{LambekJ:rings}
Joachim Lambek.
\newblock {\em Lectures on Rings and Modules}.
\newblock Blaisdell Publishing Co., 1966.
\newblock \href {https://doi.org/10.2307/2317170} {\path{doi:10.2307/2317170}}.

\bibitem{LambekJ:grammar97}
Joachim Lambek.
\newblock Type grammar revisited.
\newblock In A.~Lecomte et~al, editor, {\em Logical Aspects of Computational
  Linguistics {(LACL)} '97}, volume 1582 of {\em Lecture Notes in Computer
  Science}, pages 1--27. Springer, 1997.
\newblock \href {https://doi.org/10.1007/3-540-48975-4\_1}
  {\path{doi:10.1007/3-540-48975-4\_1}}.

\bibitem{LambekJ:Galois}
Joachim Lambek.
\newblock Iterated galois connections in arithmetic and linguistics.
\newblock In K.~Denecke et~al, editor, {\em Galois Connections and
  Applications}, Mathematics and Its Applications, pages 389--397. Springer,
  2004.
\newblock \href {https://doi.org/10.1007/978-1-4020-1898-5_11}
  {\path{doi:10.1007/978-1-4020-1898-5_11}}.

\bibitem{LambekJ:word-sentence}
Joachim Lambek.
\newblock {\em From Word to Sentence: A Computational Algebraic Approach to
  Grammar}.
\newblock Open access publications. Polimetrica, 2008.

\bibitem{Lambek-Moser}
Joachim Lambek and Leo Moser.
\newblock Inverse and complementary sequences of natural numbers.
\newblock {\em The American Mathematical Monthly}, 61(7):454--458, 1954.
\newblock \href {https://doi.org/10.2307/2308078} {\path{doi:10.2307/2308078}}.

\bibitem{LawvereFW:dialectica}
F.~William Lawvere.
\newblock Adjointness in foundations.
\newblock {\em Dialectica}, 23:281--296, 1969.
\newblock reprint in Theory and Applications of Categories, No. 16, 2006,
  pp.1--16.
\newblock \href {https://doi.org/10.1111/j.1746-8361.1969.tb01194.x}
  {\path{doi:10.1111/j.1746-8361.1969.tb01194.x}}.

\bibitem{LawvereFW:equhcs}
F.~William Lawvere.
\newblock Equality in hyperdoctrines and the comprehension schema as an adjoint
  functor.
\newblock In Alex Heller, editor, {\em Applications of Categorical Algebra},
  number~17 in Proceedings of Symposia in Pure Mathematics, pages 1--14.
  American Mathematical Society, 1970.
\newblock \href {https://doi.org/10.1090/pspum/017}
  {\path{doi:10.1090/pspum/017}}.

\bibitem{LewisM:NN}
Martha Lewis.
\newblock Compositionality for recursive neural networks.
\newblock {\em J. Appl. Logics}, 6(4):709--724, 2019.

\bibitem{CoeckeB:QNLP21}
Robin Lorenz, Anna Pearson, Konstantinos Meichanetzidis, Dimitri Kartsaklis,
  and Bob Coecke.
\newblock {QNLP in Practice: Running Compositional Models of Meaning on a
  Quantum Computer}.
\newblock {\em CoRR}, abs/2102.12846, 2021.
\newblock \href {http://arxiv.org/abs/2102.12846} {\path{arXiv:2102.12846}}.

\bibitem{MarkovA:onegin}
Andrei~A. Markov.
\newblock An example of statistical investigation in the text of {Eugene
  Onyegin}, illustrating coupling of tests in chains.
\newblock {\em Proceedings of the Academy of Sciences of St.~Petersburg},
  VI(7):153--162, 1913.
\newblock \href {https://doi.org/10.1017/S0269889706001074}
  {\path{doi:10.1017/S0269889706001074}}.

\bibitem{CoeckeB:QNLP20}
Konstantinos Meichanetzidis, Stefano Gogioso, Giovanni de~Felice, Nicol{\`{o}}
  Chiappori, Alexis Toumi, and Bob Coecke.
\newblock {Quantum Natural Language Processing on Near-Term Quantum Computers}.
\newblock {\em CoRR}, abs/2005.04147, 2020.
\newblock \href {http://arxiv.org/abs/2005.04147} {\path{arXiv:2005.04147}}.

\bibitem{MoortgatM:Dyck}
Michael Moortgat.
\newblock A note on multidimensional dyck languages.
\newblock In Claudia Casadio, Bob Coecke, Michael Moortgat, and Philip Scott,
  editors, {\em Categories and Types in Logic, Language, and Physics - Essays
  Dedicated to Jim Lambek on the Occasion of His 90th Birthday}, volume 8222 of
  {\em Lecture Notes in Computer Science}, pages 279--296. Springer, 2014.
\newblock \href {https://doi.org/10.1007/978-3-642-54789-8\_16}
  {\path{doi:10.1007/978-3-642-54789-8\_16}}.

\bibitem{Mehrnoosh-Moortgat:parasitic}
Michael Moortgat, Mehrnoosh Sadrzadeh, and Gijs Wijnholds.
\newblock A {Frobenius} algebraic analysis for parasitic gaps.
\newblock {\em {FLAP}}, 7(5):823--852, 2020.
\newblock URL: \url{http://collegepublications.co.uk/ifcolog/?00041}.

\bibitem{NaurP:algol}
Peter Naur.
\newblock A course of algol 60 programming.
\newblock {\em ALGOL Bull.}, (Sup 9):1–38, jan 1961.

\bibitem{PavlovicD:mapsII}
Dusko Pavlovic.
\newblock Maps {II}: Chasing diagrams in categorical proof theory.
\newblock {\em J. of the IGPL}, 4(2):1--36, 1996.
\newblock \href {https://doi.org/10.1093/jigpal/4.2.159}
  {\path{doi:10.1093/jigpal/4.2.159}}.

\bibitem{PavlovicD:QI09}
Dusko Pavlovic.
\newblock Quantum and classical structures in nondeterministic computation.
\newblock In Peter Bruza, Don Sofge, and Keith {van Rijsbergen}, editors, {\em
  Proceedings of Quantum Interaction 2009}, volume 5494 of {\em Lecture Notes
  in Artificial Intelligence}, pages 143--158. Springer Verlag, 2009.
\newblock \href {https://doi.org/10.1007/978-3-642-00834-4_13}
  {\path{doi:10.1007/978-3-642-00834-4_13}}.

\bibitem{PavlovicD:Qabs12}
Dusko Pavlovic.
\newblock Geometry of abstraction in quantum computation.
\newblock {\em Proceedings of Symposia in Applied Mathematics}, 71:233--267,
  2012.
\newblock arxiv.org:1006.1010.
\newblock \href {https://doi.org/10.1090/psapm/071/607}
  {\path{doi:10.1090/psapm/071/607}}.

\bibitem{PavlovicD:ICFCA12}
Dusko Pavlovic.
\newblock {Quantitative Concept Analysis}.
\newblock In Florent Domenach, Dmitry~I. Ignatov, and Jonas Poelmans, editors,
  {\em Proceedings of ICFCA 2012}, volume 7278 of {\em Lecture Notes in
  Artificial Intelligence}, pages 260--277. Springer Verlag, 2012.
\newblock \href {https://doi.org/10.1007/978-3-642-29892-9_24}
  {\path{doi:10.1007/978-3-642-29892-9_24}}.

\bibitem{PavlovicD:Samson13}
Dusko Pavlovic.
\newblock {Bicompletions of distance matrices}.
\newblock In Bob Coecke, Luke Ong, and Prakash Panangaden, editors, {\em
  {Computation, Logic, Games and Quantum Foundations. The Many Facets of Samson
  Abramsky}}, volume 7860 of {\em Lecture Notes in Computer Science}, pages
  291--310. Springer Verlag, 2013.
\newblock \href {https://doi.org/10.1007/978-3-642-38164-5_20}
  {\path{doi:10.1007/978-3-642-38164-5_20}}.

\bibitem{PavlovicD:IC12}
Dusko Pavlovic.
\newblock Monoidal computer {I}: {Basic computability by string diagrams}.
\newblock {\em Information and Computation}, 226:94--116, 2013.
\newblock \href {https://doi.org/10.1016/j.ic.2013.03.007}
  {\path{doi:10.1016/j.ic.2013.03.007}}.

\bibitem{PavlovicD:MonCom}
Dusko Pavlovic.
\newblock {\em {Computer Science in Diagrams}}.
\newblock textbook manuscript, in process by publisher, 2021.

\bibitem{PavlovicD:nucleus}
Dusko Pavlovic and Dominic~J.D. Hughes.
\newblock {The nucleus of an adjunction and the Street monad on monads}.
\newblock {\em CoRR}, abs/2004.07353:87 pages, 2020.
\newblock submitted.
\newblock URL: \url{http://arxiv.org/abs/2004.07353}, \href
  {http://arxiv.org/abs/2004.07353} {\path{arXiv:2004.07353}}.

\bibitem{PavlovicD:MonCom3}
Dusko Pavlovic and Muzamil Yahia.
\newblock {Monoidal computer {III}: {A coalgebraic view of computability and
  complexity}}.
\newblock In Corina C{\^{\i}}rstea, editor, {\em Coalgebraic Methods in
  Computer Science ({CMCS}) 2018 --- Selected Papers}, volume 11202 of {\em
  Lecture Notes in Computer Science}, pages 167--189. Springer, 2018.
\newblock \href {https://doi.org/10.1007/978-3-030-00389-0}
  {\path{doi:10.1007/978-3-030-00389-0}}.

\bibitem{PlotkinG:SOS}
Gordon~D Plotkin.
\newblock A structural approach to operational semantics.
\newblock {\em The Journal of Logic and Algebraic Programming}, 60-61:3--139,
  2004.
\newblock lecture notes from 1981, circulated as Tech. Rep. DAIMI FN-19,
  Computer Science Department, Aarhus University.

\bibitem{PrellerA:linear}
Anne Preller.
\newblock Linear processing with pregroups.
\newblock {\em Stud. Logica}, 87(2-3):171--197, 2007.
\newblock \href {https://doi.org/10.1007/s11225-007-9087-0}
  {\path{doi:10.1007/s11225-007-9087-0}}.

\bibitem{Mehrnoosh:word-one}
Mehrnoosh Sadrzadeh, Stephen Clark, and Bob Coecke.
\newblock The {Frobenius} anatomy of word meanings {I:} subject and object
  relative pronouns.
\newblock {\em J. Log. Comput.}, 23(6):1293--1317, 2013.

\bibitem{Mehrnoosh:word-two}
Mehrnoosh Sadrzadeh, Stephen Clark, and Bob Coecke.
\newblock The {Frobenius} anatomy of word meanings {II:} possessive relative
  pronouns.
\newblock {\em J. Log. Comput.}, 26(2):785--815, 2016.

\bibitem{Esquisse}
L.~Schneps and P.~Lochak, editors.
\newblock {\em Geometric Galois Actions. 1. Around Grothendieck's Esquisse D'un
  Programme}, volume 242 of {\em London Mathematical Society Lecture Note
  Series}.
\newblock Cambridge University Press, 1997.
\newblock \href {https://doi.org/10.1017/CBO9780511758874.002}
  {\path{doi:10.1017/CBO9780511758874.002}}.

\bibitem{StreetR:Frobenius}
Ross Street.
\newblock Frobenius monads and pseudomonoids.
\newblock {\em Journal of mathematical physics}, 45(10):3930--3948, 2004.
\newblock \href {https://doi.org/10.1063/1.1788852}
  {\path{doi:10.1063/1.1788852}}.

\bibitem{ThueA:selected}
Axel Thue, C.L. Siegel, and T.~Nagell.
\newblock {\em Selected mathematical papers of Axel Thue}.
\newblock Universitetsforlaget, 1977.
\newblock \href {https://doi.org/10.1090/S0002-9904-1978-14535-2}
  {\path{doi:10.1090/S0002-9904-1978-14535-2}}.

\bibitem{Ward-Dilworth}
Morgan Ward and Robert~P Dilworth.
\newblock Residuated lattices.
\newblock {\em Transactions of the American Mathematical Society},
  45(3):335--354, 1939.
\newblock \href {https://doi.org/10.1090/S0002-9947-1939-1501995-3}
  {\path{doi:10.1090/S0002-9947-1939-1501995-3}}.

\bibitem{WillertonS:nucleus}
Simon Willerton.
\newblock Tight spans, {Isbell} completions and semi-tropical modules.
\newblock {\em Theory and Applications of Categories}, 28(22):696--732, August
  2013.

\end{thebibliography}

\end{document}